\documentclass[11pt]{article}
\usepackage{amssymb}
\usepackage{amsmath}
\RequirePackage{hyperref}

\usepackage{algorithmic}
\usepackage{algorithm}

\usepackage{color}

\newenvironment{proof}{\noindent {\bf Proof }}
{\hfill $\bullet$ \vspace{0.25cm}}

\def\R{{\mathbb R}}

\def\N{{\mathbb N}}

\def\F {{\mathcal F}}

\def\TT{{\mathcal T}}

\def\s {{\sigma}}
\def\la {{\lambda}}
\def\eps {{\epsilon}}

\newcommand{\dis}{\displaystyle}

\newtheorem{theo}{Theorem}
\newtheorem{prop}{\indent Proposition}

\newtheorem{rem}{\indent Remark}
\newtheorem{lem}{\indent Lemma}
\newtheorem{defin}{\indent Definition}
\newtheorem{cor}{\indent Corollary}

\newtheorem{ass}{\indent Assumption}

\title{Hydrodynamic limit for interacting neurons}

\author{A.~De Masi\thanks{e-mail addresses: \tt{anna.demasi@gmail.com}, \tt{galves@usp.br},
    \tt{eva.loecherbach@u-cergy.fr} and
    \tt{errico.presutti@gmail.com}}, A.~Galves, E.~L\"ocherbach
  and E.~Presutti \vspace{.5cm} \\
{\it Universit\`a degli Studi dell'Aquila, Universidade de S\~ao Paulo,} \\
 {\it Universit\'e de Cergy-Pontoise and GSSI, L'Aquila}}

\date{October 4, 2014}

\begin{document}

\maketitle

\begin{abstract}
This paper studies the hydrodynamic limit of a stochastic process
  describing the time evolution of a system with $N$ neurons with
  mean-field interactions produced both by chemical and by electrical
  synapses. This system can be informally described as follows. Each
  neuron spikes randomly following a point process with rate depending
  on its membrane potential.  At its spiking time, the membrane
  potential of the spiking neuron is reset to the value
  $0$ and, simultaneously, the membrane potentials of
  the other neurons are increased by an amount of
    potential $\frac{1}{N} $. This mimics the effect of chemical
  synapses.  Additionally, the effect of electrical synapses is
  represented by a deterministic drift of all the membrane potentials
  towards the average value of the system.

  We show that, as the system size $N$ diverges, the distribution of
  membrane potentials becomes deterministic and is described by a
  limit density which obeys a non linear PDE which is a conservation
  law of hyperbolic type.
\end{abstract}



\section{Introduction}
This paper studies the hydrodynamic limit of a continuous time stochastic process describing a system of interacting neurons. The system we consider is made of $N$ neurons whose state is specified by $U^N (t) = (U^N_1 (t) , \ldots , U^N_N ( t) )$, $t \geq 0$, $U^N (t) \in \R_+^N, $ for some fixed integer $N \geq 1 .$ Each $U^N_i (t) $ models the membrane potential of neuron $i$ at time $t, $ for $i = 1 , \ldots , N .$ Neurons interact either by {\sl chemical}  or by {\sl electrical} synapses. Our model does not consider external stimuli.

Chemical synapses can be described as follows. Each neuron spikes randomly following a point process with rate depending on the membrane potential of the neuron.  At its spiking time, the membrane potential of the spiking neuron is reset to a reversal potential. At the same time, simultaneously, the other neurons, which do not spike, receive an additional amount of potential $\frac{1}{N} $ which is added to their membrane potential.

Electrical synapses occur through {\sl gap-junctions} which allow
neurons in the brain to communicate directly. This induces an
attraction between the values of the membrane potentials and, as a consequence, a drift of the system towards its average membrane potential.

Our model is a continuous time version of a new class of biological neuronal systems introduced recently by Galves and L\"ocherbach (2013)\nocite{antonioeva13}. The model considered in Galves and L\"ocherbach (2013) is a non Markovian system consisting of an infinite number of interacting chains where each component has memory of variable length and where each neuron is represented through its spike train. In the present paper we add gap junctions to this system and we adopt an equivalent description via the membrane potential of each neuron, leading to a Markovian process.

The number of neurons in the brain is huge and often neurons have similar properties (see Gerstner and Kistler (2002), Chapter 1.5.1). Therefore we assume that we are in an idealized situation where all neurons have identical properties, leading to a mean field description. The mean field assumption appears in the following aspects. For the chemical synapses it is translated into the fact that when a neuron spikes the membrane potential of any other neuron increases by $1/N$. For the electrical synapses, the mean field type assumption implies that the drift felt by each neuron potential is described by a linear attraction towards the average membrane potential of the system.

We regard the state of the neurons $U^N(t)=(U^N_1(t), \ldots , U^N_N(t))$ as a distribution of $1/N$ valued Dirac {\sl masses} placed at the {\sl positions} $U^N_1(t), \ldots ,U^N_N(t)$.
The main result of the present paper, presented in Theorem \ref{theo:main}, is that in the limit as $N\to \infty$
this membrane potential distribution becomes deterministic and it is described by a density $\rho_t(r)$. More precisely, in the limit, for any  interval $I \subset \mathbb R _+$,
$\int_I\rho_t(r)dr$ is the limit fraction of neurons whose membrane potentials are in $I$
at time $t$.  The limit density  $\rho_t(r)$ is proved to obey a non linear PDE which is a conservation law of
hyperbolic type.

The usual approach to prove hydrodynamic limits in mean field systems is to show that propagation of chaos holds. In our case this amounts to prove that the membrane potentials $ U^N_i (t) $ and $U^N_j (t) $ of any pair $i$ and $j$
of neurons get uncorrelated as $ N \to \infty.$
However, at each time that another neuron fires, it instantaneously affects both $U^N_i $ and $U^N_j $ by changing them with an additional amount $ 1/ N .$ Thus $U^N_i$ and $U^N_j $ are correlated, and propagation of chaos comes only by proving first that the firing activity of the other neurons -- by propagation of chaos -- is essentially deterministic. We are thus caught in a circular argument and it is not clear {\sl a priori}  that propagation of chaos holds. It is for this reason that in this paper we introduce an auxiliary process $Y^{(\delta)} $ which is a good approximation of the true process in the $ N \to \infty $ limit, and for which it is easy to prove the hydrodynamic limit.  Once the convergence for $Y^{(\delta)} $ is proved, we can then conclude by letting $ \delta \to 0 .$

Our model is an example of the class of processes introduced by Davis (1984)\nocite{davis84} under the name of {\sl piecewise deterministic Markov processes}. Processes in this class combine a deterministic continuous motion (in our case, due to the electrical synapses) with discontinuous, random jump events (in our case, the spike events). This is not the first time  that piecewise deterministic Markov processes are used in the modelization of neuronal systems, see for instance Pakdaman et al.\ (2010)\nocite{Pakdaman2010} and Riedler et al.\ (2012)\nocite{michele} in which processes of this type appear, however in a different context.

The mean field approach intending to replace individual behavior in large homogeneous systems of interacting neurons by the mean behavior of the neuronal population has a long tradition in the frame of neural networks, see e.g.\ Chapter 6 of Gerstner and Kistler (2002)\nocite{Gerstner} or Faugeras et al.\ (2009)\nocite{Faugerasetal2009} and the references therein. Most of the models used in the literature are either based on rate models where randomness comes in through random synaptic weights (see e.g.\ Cessac et al.\ (1994)\nocite{Cessac1994} or Moynot and Samuelides (2002)\nocite{Moynot02}); or they are based on populations of integrate and fire neurons which are diffusion models in either finite or infinite dimension, see for instance Delarue et al.\ (2012)\nocite{delarue} or Touboul (2014)\nocite{Touboul2014}.  The model we consider is reminiscent of integrate-and-fire models but firing does not occur when reaching a fixed threshold, and the membrane potential is not described by a diffusion process. In particular, the equation which we obtain is different from usual population density equations obtained for integrate-and-fire neurons as considered e.g.\ in Chapter 6.2.1 of Gerstner and Kistler (2002).

Our paper is organized as follows. In Section \ref{sec:def} we
introduce the process and state the main results, Theorem \ref{thm1.1}, Theorem \ref{theo:main} and Theorem \ref{theo:main2}.  Theorem \ref{thm1.1} guarantees the
existence of the process and gives upper bounds
on the values of the potentials $U^N$ which are uniform in $N$.
Theorems \ref{theo:main} and \ref{theo:main2} give existence
and properties of the hydrodynamic limit.

Proofs are organized as follows:
we first study the system under very restrictive assumptions on the
firing rate $f$, in such a case the proof of  Theorem \ref{thm1.1} becomes
trivial and is given in Section \ref{sec:3bis}.  Even with such an assumption
on $f$ the proof of Theorems \ref{theo:main} and \ref{theo:main2} remains rather complex.
In Section \ref{sec:2},
tightness of the sequence of processes indexed by $N$ is proved.
Section \ref{sec:ydelta} introduces the sequence of auxiliary
processes, and Section \ref{sec:hy} states the hydrodynamic limit
theorem for this sequence; the proof is postponed to Appendix \ref{sec:prooftheo2}.  Section \ref{sec:6} concludes the proof of
Theorems \ref{theo:main} and \ref{theo:main2}. In the Appendix we extend the result to general firing
rates $f.$ The main point is the proof of Theorem \ref{thm1.1} which is given in Appendix \ref{sec:proof1}, together with some upper bounds for the maximal membrane potential of the process in the case of unbounded firing rate functions. In Appendix \ref{sec:proofcoupling} we prove that the auxiliary process is close to the original one, if both are suitably coupled. Finally, in Appendix \ref{sec:prooftheo2}, the hydrodynamic limit for the auxiliary process is rigorously proved.

\section{Model definition and main results}\label{sec:def}
We consider a Markov process
$$U^N (t) = (U^N_1 (t), \ldots , U^N_N ( t) )   ,  \, t \geq 0 , $$
taking values in $\R_+^N, $ for some fixed integer $N \geq 1 ,$
whose generator is given for any smooth test function $ \varphi : \R_+^N \to \R $ by
\begin{equation}\label{eq:generator0}
L \varphi (x ) = \sum_{ i = 1 }^N f(x_i) \left[ \varphi (x + \Delta_i ( x)  ) - \varphi (x) \right]
- \lambda \sum_i \left( \frac{\partial \varphi}{\partial x_i} (x) \left[  x_i - \bar x  \right]  \right) ,
\end{equation}
where
\begin{equation}
(\Delta_i (x))_j =    \left\{
\begin{array}{ll}
\frac1N & j \neq i \\
- x_i & j = i
\end{array}
\right\} , \quad \bar x = \frac1N \sum_{i=1}^N x_i
\end{equation}
and where $\lambda \geq 0 $ a positive parameter. Assume that

\begin{ass}\label{ass:1}
$f\in C^1(\mathbb R_+,\mathbb R_+)$ is strictly positive for $x>0$ and non-decreasing. Moreover, $f(0) = 0 $ and $f$ is not flat, i.e. for any fixed $u \in ]0, 1 [ ,$
$$ \lim \inf_{ x \to 0} \frac{f(ux) }{f(x) } > 0.$$
\end{ass}
The function $ f(x) = x^p ,$ $p> 0 , $ satisfies the above assumption. We can also consider functions $f(x) = e^{\nu x } - 1, $ for some $\nu > 0.$

In (\ref{eq:generator0}), the first term describes random jumps  at rate $f (x_i)$ due to spiking of neurons having potential $x_i$.  The function $f$ is therefore called firing rate or spiking rate of the system. The second term, due to electrical synapses (gap junctions), describes a deterministic time evolution tending to attract the neurons to the common average potential.

Our first theorem proves the existence of the process and gives some a priori estimates on the maximal membrane potential. In order to state these results, we introduce the following notation.
 Let  $N_i (t) , t \geq 0 , $  be the simple point process on $\R_+ $ which counts the jump events of neuron $i$ up to time $t$ and let
\begin{equation}\label{eq:nt}
  N(t) = \sum_{i=1}^N N_i(t)
\end{equation}
be the total number of jumps seen before time $t$. For any $ x \in \R^N , $ we define $\| x \| = \max_{ i = 1 , \ldots , N } x_i .$ In this way,
$$
\| U^N(t)\|=\max_{i=1,\ldots , N} U^N_i(t)
$$
is the maximal membrane potential at time $t.$

\begin{theo}\label{thm1.1}
Let $f$ be a firing rate function satisfying Assumption \ref{ass:1}.
\begin{enumerate}
\item
For any $N \geq 1 $ and any $x \in \R_+^N $ there exists a unique strong Markov process $U^N(t) $ taking values in $\R_+^N $ starting from $x$    whose generator  is given by \eqref{eq:generator0}.

\item
Denote by $P_{ x}^{(N,\lambda)}$ the probability law under which the process $U^N(t)$ starts from the initial configuration $ U^N(0) = x= (x_1, \ldots , x_N) \in \R_+^N .$ Then for any $A > 0$ and $T> 0 $ there exists $B$ such that
\begin{equation}
\label{1.1}
 \sup_{ x: \| x\|\le A} P^{(N,\lambda)}_{ x}\Big[ \sup_{t\le T }\|U^N(t) \| < B\Big]  \geq 1 - c e^{ - C N } ,
\end{equation}
where $c$ and $C$ are suitable constants.
\end{enumerate}
\end{theo}

The proof of Theorem \ref{thm1.1} is given in the Appendix \ref{sec:proof1}.

We now give the main result of this paper. It shows that the process converges in the hydrodynamic limit, as $ N \to \infty , $ to a specified evolution which will be defined below. Since the space where $ U^N(t) $ takes values changes with $N$ it is  convenient to identify configurations $U^N(t) $ with the associated empirical measure in the following way. Let ${\cal M} $ be the space of all probability measures on $ \R_+ .$ To any $x = (x_1, \ldots , x_N)  \in \R_+^N$ we associate the element of $\mathcal M$ given by
\begin{equation}
\label{eq:emp}
\mu_{ x}  = \frac 1N \sum_{i=1}^N \delta_{x_i} .
\end{equation}
$\mu_{x}$ has the nice physical-biological interpretation of being the
distribution of membrane potentials of the neurons.

We suppose that for all $N, $ $U_i^N (0) = x_i^{N} ,$ $i = 1 , \ldots , N , $ such that the following assumption is satisfied.

\begin{ass}\label{ass:0}
$x^N_1, \ldots , x^N_N $ are i.i.d.\ random variables, distributed according to $\psi_0 (x) dx $ on $\R_+ .$
Here, $\psi_0$ is a smooth probability density on $ \R_+ $ with
compact support $ [ 0, R_0 ] $ such that the following properties are verified.
\begin{enumerate}
\item
$\psi_0 > 0 $ on $ [0, R_0 [ .$
\item
$\psi_0  \equiv 0 $ on $ [ R_0 , \infty [ .$
\item
$\psi_0 (x) \geq c ( x - R_0 ) ^2 , $ $c> 0 , $ in a left neighborhood of $R_0 .$
\end{enumerate}
\end{ass}
The above assumption can be weakened, see Remark \ref{rem:relaxed} below. Condition 3.\ could be relaxed to other rates of decay to $0$ near $R_0 .$
We will eventually extend the definition of $ \psi_0 $ to the whole line by putting $ \psi_0 ( x) = \psi_0 (0) $ for all $x < 0 .$

Identifying $ U^N (t) $ with the associated probability measure $\mu_{ U^N(t) },$ we may identify the process with the element $ \R_+ \ni t \to \mu_{ U^N(t) } $  of the Skorokhod space $ D ( \R_+ , {\cal S}') ,$ where $ {\cal S}$ is the Schwartz space of all smooth functions $\phi : \R \to \R  .$ We write $ \mu_{U^N_{[ 0 , T ] } }$ for the restriction of this process to $ [ 0, T ] $ which is an element of $ D ( [0, T ]  , {\cal S}') .$ Our next theorem states that $\mu_{U^N_{[ 0 , T ] } }$ converges to a deterministic limit density $(\rho_t(x)dx )_{ t \in [0, T ] } $.
We can easily guess the equation satisfied by
$\rho_t(x)$. In fact if $\rho_t(x)$ is the limit density then
the limit total firing rate per unit time $p_t$  and  the limit average membrane potential $\bar \rho_t$ are
\begin{equation}
\label{eq:217.1}
 p_t = \int_0^\infty f(x) \rho_t (x) dx, \;\;\;\; \bar \rho_t = \int_0^\infty x \rho_t (x) dx .
 \end{equation}
Thus
 \begin{equation}
 \label{eq:217.2}
 V(x,\rho_t):=  -\lambda  (x - \bar\rho_t)+ p_t
  \end{equation}
is the velocity field, namely the limit drift that
neurons have at time $t$ and at energy $x$,
the first term being the attraction to the average membrane potential of the system, due to the gap junction effect,
the second one the drift produced by the
other neurons spiking.
Besides such a mass transport
we have also a  loss of mass term $f(x) \rho_t(x)$ due to spiking
so that we should expect that for smooth $\rho_t(x)$
  \begin{equation}
 \label{eq:217.3}
\frac{ \partial }{ \partial t}\rho_t + \frac{ \partial }{ \partial x}(V\rho_t) = - f \rho_t,\quad x>0,t>0 .
   \end{equation}
However
\eqref{eq:217.3} does not determine the solution, it must be complemented by
boundary conditions:
  \begin{equation}
 \label{eq:217.3.1}
\rho_0(x)= u_0(x),\quad  \rho_t(0)= u_1(t) .
   \end{equation}
$u_0$ is  specified by the problem: $u_0=\psi_0$, $u_1$ instead must be derived
together with \eqref{eq:217.3}.  It turns out from our analysis that
  \begin{equation}
 \label{eq:217.3.2}
u_1(t) = \frac{p_t}{V(0,\rho_t)} = \frac{p_t}{p_t+ \la \bar\rho_t} .
   \end{equation}
\eqref{eq:217.3.2} follows from conservation of mass as it will be
discussed
after the  definition of weak solutions of \eqref{eq:217.3}--\eqref{eq:217.3.1}.
Indeed if $u_0(0)\ne u_1(0)$, i.e.\ $\psi(0) \ne  \frac{p_0}{V(0,\psi_0)}$, then $\rho_t(x)$ cannot be continuous, hence the necessity of a weak formulation of
\eqref{eq:217.3}--\eqref{eq:217.3.1}.

\begin{defin}
A real valued function $ \rho_t (x)$ defined on $(t, x ) \in \R_+ \times \R_+ $ is a weak solution of \eqref{eq:217.3}--\eqref{eq:217.3.1} if
for all smooth functions $\phi(x)$, $\mathbb R_+\ni t \to  \int \phi(x)\rho_t( x)dx $
is continuous, differentiable in $t > 0 $ and
\begin{eqnarray}
\label{eq:limitpde}
&& \frac d{dt} \int_0^\infty \phi ( x)  \rho_t(x) dx  -  \int_0^\infty \phi'  ( x) V(x,\rho_t)
\rho_t (x)dx
- \phi (0 ) V(0,\rho_t) u_1(t) \nonumber\\&&\hskip3cm = -
 \int_0^\infty  \phi ( x)  f(x) \rho_t (x) dx , \nonumber
 \\
 \\&& \int_0^\infty \phi ( x)  \rho_0(x) dx = \int_0^\infty \phi ( x)  u_0(x) dx ,\nonumber
\end{eqnarray}
where $V(x,\rho_t)$ is given by \eqref{eq:217.2} with $p_t$ and $\bar\rho_t$
as in \eqref{eq:217.1}.

\end{defin}

Let us now give a heuristic derivation
of \eqref{eq:217.3.2}. Observe that if  $\rho_t$ is
the limit density of our neuron system then, by definition, at all times $t\ge 0$
 \begin{equation}
 \label{eq:217.4}
 \int_0^\infty \rho_t (x) dx = 1.
   \end{equation}
Recalling that $V(x,\rho_t)$ is the limit velocity field, we have that
the rate at which  mass enters into $(0,\infty)$ is $V(0,\rho_t)u_1(t)$
while the rate at which  mass leaves  $(0,\infty)$ is $p_t$ (due to spiking).
Mass conservation then indicates that
$V(0,\rho_t)u_1(t)=p_t$ for almost all $t$, hence  \eqref{eq:217.3.2}.

As we shall see in the next theorem the limit density solves \eqref{eq:limitpde}
and it can be quite explicitly computed by using the method of characteristics. The
characteristics
are curves along which the solution is transported, they are
defined by the equation
   \begin{eqnarray}
   \label{eq:217.5}
&&\frac{dx(t)  }{d t } = V(x(t),\rho_t) .
   \end{eqnarray}
The solution of \eqref{eq:217.5} in the time interval $[s,t]$, $ 0 \le s \le t ,$
with value $x$ at time $s$ is denoted by $\varphi_{s,t}(x)$, $x \in \R_+$,
and it has the following expression:
\begin{equation}\label{eq:flowlimit}
\varphi_{s, t } (x) = e^{ - \lambda (t-s) } x +  \int_s^t e^{ - \lambda (t-u) }[  \lambda \bar \rho_u +p_u] du .
\end{equation}
Now our main result reads as follows.

\begin{theo}
\label{theo:main}

Grant Assumptions \ref{ass:1} and \ref{ass:0}.
For any fixed $ T > 0 , $
\begin{equation}
{\cal L} (\mu_{ U^{N}_{ [0, T ] }} ) \stackrel{w}{\to} {\cal P}_{[0, T ] }
\end{equation}
(weak convergence in $ D ( [0, T ]  , {\cal S}') $) as $N \to \infty ,$ where $ {\cal P}_{ [0, T ] } $ is the law on $  D ( [0, T ] , {\cal S}') $ supported by the distribution valued trajectory $ \omega_t $ given by
$$ \omega_t ( \phi ) = \int_0^\infty  \phi (x) \rho_t (x) dx , \quad  t \in [0, T ] ,$$
for all $\phi \in {\cal S}.$

Here, $\rho_t (x) $ is the unique weak solution of  \eqref{eq:217.3}--\eqref{eq:217.3.1}
with $u_0= \psi_0$ and $u_1$ as in \eqref{eq:217.3.2}.
Moreover, $\rho_t ( x) $ is a continuous function of $(x,t)$ in
$\mathbb R_+\times \mathbb R_+ \setminus \{(\varphi_{0,t}(0),t), t \in \mathbb R_+\}$ where it
is
differentiable in $x$ and $t$ and the derivatives satisfy
\eqref{eq:217.3}.  Moreover
for any $t\ge 0$,  $\rho_t ( x) $  has compact support
in $x$  and
\begin{equation}
\label{eq:initial0}
 \rho_t(0 ) = \frac{ p_t}{p_t  + \lambda \bar \rho_t},\quad \int \rho_t(x) dx =1   .
\end{equation}
Its explicit expression for  $x \geq \varphi_{0, t }(0)$ is:
\begin{equation}\label{eq:ut-0}
\rho_t (x) =  \psi_0 \left(\varphi_{0, t }^{ - 1 } (x) \right)
\exp \left\{ - \int_0^t [ f - \lambda ] ( \varphi_{ s, t }^{ - 1 } (x) ) ds   \right\} ,
\end{equation}
and for any $ x = \varphi_{s, t} (0)  $ for some $ 0 < s \le t , $
\begin{equation}\label{eq:ut-1}
 \rho_t (x) = \frac{p_s}{ p_s + \lambda \bar \rho_s  } \exp \left\{ - \int_s^t [f ( \varphi_{s, u } ( 0 ) ) - \lambda ] du \right\} .
\end{equation}

\end{theo}


\begin{theo}\label{theo:main2}
Grant Assumptions \ref{ass:1}, \ref{ass:0} and suppose that
\begin{equation}\label{eq:border}
\psi_0(0)= \frac{ p_0 }{ p_0 + \lambda \bar \psi_0 }, \mbox{ where $ p_0 = \int_0^\infty f (x) \psi_0 (x) dx $ and $ \bar \psi_0 = \int_0^\infty x \psi_0 (x) dx.$}
\end{equation}
Then
$\rho_t(x) $ is continuous in $\mathbb R_+\times\mathbb R_+$.
%
\end{theo}

We give some comments on the above result. We first compare our result with classical ``population density equations" obtained in integrate-and-fire models as for instance described in Gerstner and Kistler (2002). In our second remark, we discuss condition \eqref{eq:border}.

\begin{rem}
In case $\lambda = 0 , $ \eqref{eq:217.3} reads as follows.
$$ \left\{
\begin{array}{lcll}
\partial_t \rho_t(x)    &=&- p_t  \partial_x \rho_t(x)    - f(x)\rho_t (x) , &  x > 0 , x \neq \varphi_{0, t }(0) ,
\\
 \rho_t(0 ) &= & 1  \;&  \mbox{for all $t \geq 0$}  .
\end{array}
\right. $$
This equation is different from usual population density equations which are obtained for integrate-and-fire neurons as considered e.g.\ in Chapter 6.2.1 of Gerstner and Kistler (2002), see in particular their formula (6.14). As in integrate-and-fire models, also in our model spiking neurons are reset to a reversal potential (which equals $0$); but spiking does not create Dirac-masses at the reset value. This is due to the Poissonian mechanism giving rise to spiking in our model. The loss of mass at time $t$ due to spiking of neurons having potential height $x$ is therefore described by the term $ - f(x) \rho_t (x) .$

At the same time, spiking induces a deterministic drift $ p_t dt $ for those neurons that are not spiking. In particular, a neuron having initially potential $0$ at time $t$ will have potential $\approx p_t h $ after a time $ t + h ,$ for $h << 1$ small. Hence, during $ [t, t + h ],$ there is creation of an interval $ [0, p_t h ] $ at the beginning of the support in which no non-spiking neurons are present. At the same time, there are approximately $p_t h $ neurons that spike during $ [t, t + h ]$ which invade this initial interval. This implies that the initial density of neurons at the border $x= 0 $ is of height $1.$ This initial condition is different from the usual initial condition obtained in integrate-and-fire models.
\end{rem}

\begin{rem}
The condition (\ref{eq:border}) ensures that the limit density $\rho_t (x) $ does not have a discontinuity at the point $x = \varphi_{0, t } (0) .$ This point $\varphi_{0, t } (0)$ is the point where two densities are pieced together: on the one hand the density of neurons that did not yet spike up to time $t,$ which is given by formula (\ref{eq:ut-0}), and on the other hand the density of neurons that have already spiked, given by (\ref{eq:ut-1}). Without condition (\ref{eq:border}), the convergence result still holds true, but $ \rho_t (x) $ will have a (single) jump at $x = \varphi_{0, t } (0) ;$ in particular, it is not a strong solution of the nonlinear PDE. However, even without condition \eqref{eq:border}, for any $t > 0 , $ \eqref{eq:initial0} holds true.
\end{rem}

To separate the difficulties  we shall first prove Theorem \ref{thm1.1} and  Theorem \ref{theo:main} under a very restrictive
assumption on $f$:

\begin{ass}
 \label{ass:2}
$f$ is a positive $C^1 -$function satisfying Assumption \ref{ass:1}. $f$ is non-decreasing, Lipschitz continuous, bounded and constant for all $ x \geq x^{**} $ for some $x^{**} > 0.$  We shall denote by $f^*=\|f\|_\infty$ the sup norm of $f$.
 \end{ass}

The proof of Theorem \ref{thm1.1} under Assumption \ref{ass:2} is easy, it is given
in the next section. In the successive sections we shall prove  Theorems \ref{theo:main} and \ref{theo:main2} under
Assumption \ref{ass:2}. In Section \ref{sec:2},
tightness of the sequence of processes indexed by $N$ is proved.
In Section \ref{sec:ydelta} we introduce a sequence of auxiliary
processes which are discrete time models and for which it is easier to prove the hydrodynamical limit which is done in Section \ref{sec:hy}. Section \ref{sec:6} will then conclude the proof of
Theorems \ref{theo:main} and \ref{theo:main2} under Assumption \ref{ass:2}.

In the Appendix we shall prove Theorem \ref{thm1.1} in its original formulation (i.e.\ dropping Assumption \ref{ass:2}) and then  Theorem \ref{theo:main}.  However this last step
is trivial  because
the estimate \eqref{1.1} implies that with probability going to 1 as $N\to \infty$ all the membrane potentials
are uniformly bounded in the time interval $[0,T]$ that we are
considering.  It is then possible to replace the true $f$ with one
satisfying Assumption \ref{ass:2} and which differs only for potential values larger than
those reached by the true process, so that we can use what was already proved
under Assumption \ref{ass:2}.  The precise argument is given at the end of the Appendix.

\section{Energy bounds under Assumption \ref{ass:2}}\label{sec:3bis}

\noindent
Exploiting Assumption \ref{ass:2} we shall prove a statement stronger than in  Theorem  \ref{thm1.1}.

\begin{prop}\label{thm1.1bis}
Let $f$  satisfy Assumption \ref{ass:2} and call $f^*= \|f\|_\infty$.
\begin{enumerate}
\item
For any $N \geq 1 $ and any $x \in \R_+^N $ there exists a unique strong Markov process $U^N(t)$
starting from $x$
taking values in $\R_+^N $ whose generator  is given by \eqref{eq:generator0}.
\item
Calling $N(t)$ the total number of fires in the time interval $[0,t]$ we have
  \begin{equation}
  \label{stochbounds}
  N(t) \le N^*(t) \quad \text{stochastically}
  \end{equation}
where $N^*(t)$ is a Poisson process with intensity $Nf^* .$
\item
$\dis{\sup_ {t\le T} \|U^N(t) \| \le \| U^N(0)\|+ \frac{N(t)}{N}}$
and for any  $T> 0 $ there exist positive constants   $c$ and $C$  such that for any $N$ and any $ U^N(0)$:
\begin{equation}
\label{1.1mm}
 P^{(N,\lambda)}_{ U^N(0)}\Big[ \sup_{t\le T }\|U^N(t) \| \le \| U^N(0)\|+  2f^* T\Big]  \geq 1 - c e^{ - C T N } .
\end{equation}
\end{enumerate}
\end{prop}

\begin{proof}
The existence of the process for each fixed $N$ is
now trivial as the firing rates are bounded.
The variable
$ N(t) $ is stochastically upper bounded by $N^*(t):= \sum_{i=1}^N n_{i }(t) , $
where  $ (n_{i }(t) )$ are i.i.d.\ Poisson processes of intensity $f^* $.  $N^*(t)$ is therefore a Poisson process with intensity $Nf^*$.
We have
   \begin{equation*}
\sup_{ t \le T } \| U^N ( t)\| \le \| U^N(0)\|+ \frac{N (t)}{N},
  \end{equation*}
because each firing event increases the rightmost neuron by $ \frac{1}{N} , $ while, in between firing events,
the rightmost neuron is attracted to the average membrane potential of the process and thus decreases.
 \eqref{1.1mm} then follows from item 2.\  because $\{ N(T) \ge B\}$ is an increasing event
 and thus the bound is reduced to   large deviations for a Poisson variable, details are omitted.
\end{proof}

\section{Tightness}\label{sec:2}
With this section we begin the proof of  Theorems \ref{theo:main} and \ref{theo:main2} (under Assumption \ref{ass:2}).
We start by proving tightness of the sequence of laws of $ \mu_{ U^N_{[0, T ] }}.$

\begin{prop}\label{prop:4}
Grant Assumption \ref{ass:2}.
Suppose that $ U^N (0 ) = x^{N} $ is such that Assumption \ref{ass:0} is verified. Then the sequence of laws of $\mu_{U^N_{ [0, T ]}} $ is tight in  $ D ( \R_+ , {\cal S}') .$
\end{prop}

\begin{proof}
For any test function $\phi \in {\cal S}$ and all $t \in [0, T ], $  we write,
\begin{equation}
\langle U^N(t), \phi\rangle = \frac1N \sum_i \phi ( U^N_i (t) )  = \int \phi (x) \mu_{ U^N (t) } (dx) .
\end{equation}
By Mitoma 1983\nocite{mitoma} it is sufficient to prove the tightness of $ \langle U^N(t), \phi\rangle, t \in [0, T ]  \in D ( [0, T]  , \R )$  for any fixed $ \phi \in {\cal S} .$
In order to  do so, we shall use a well known tightness criterion, see for
instance Theorem 2.6.2 of De Masi and Presutti 1991\nocite{anna-errico91}, which requires that the $L^2$ norms of
the ``compensators'' of $\langle U^N(t), \phi\rangle$ are finite.  The compensators are
 \begin{equation}
\gamma_1^N (t) = L \langle U^N(t), \phi\rangle,\;\;  \gamma_2^N (t) =  L
 \langle U^N(t), \phi\rangle ^2 - 2  \langle U^N(t), \phi\rangle  L \langle U^N(t), \phi\rangle ,
\end{equation}
where $L$ is the generator given by \eqref{eq:generator0}.
The criterion requires that there exists
a constant  $c$ so that
 \begin{equation}
 \label{conditions}
\sup_{t \le T } E [ \gamma_1^N (t) ]^2  \le c,\;\; \;\;\sup_{t \le T } E [ \gamma_2^N (t) ]^2  \le c .
\end{equation}
The proof of the criterion is based on the fact that
  $$
M_t^N = \langle U^N(t), \phi\rangle - \int_0^t \gamma_1^N (s) ds  \mbox{ and } (M_t^N )^2
- \int_0^t \gamma_2^N (s) ds
  $$
are martingales. To prove \eqref{conditions} we start by calculating $\gamma_1^N (t)  = \frac1N \sum_i L \phi (U^N_i (t)) .$
We have
\begin{multline*}
\gamma_1^N (t)  = \\
\frac1N \sum_i  \left[ \sum_{ j \neq i } f(U^N_j (t)) [ \phi ( U^N_i (t) + \frac1N ) - \phi ( U^N_i (t))] +  f( U^N_i (t)) [ \phi ( 0) - \phi ( U^N_i (t)) ]\right] \\
+\frac{\lambda}{N}  \sum_{ i } \phi'  ( U^N_i (t)) [ \bar U_N (t)  - U^N_i (t) ] ,
\end{multline*}
where $\bar U_N(t)=\langle U^N(t),id\rangle$ is the average of the $U^N_i(t)$.
Expanding the discrete derivative, we get
\begin{multline*}
\gamma_1^N (t) = \langle U^N(t),f\rangle \langle U^N(t), \phi'\rangle
- \langle U^N(t), f \phi\rangle + \phi ( 0) \langle U^N(t),f\rangle \\
+\lambda \left[ \langle U^N(t), \phi'\rangle \langle U^N(t),id\rangle -  \langle U^N(t),\psi\rangle
\right] + O ( \frac1N) ,
\end{multline*}
where $\psi(x) = x\phi'(x)$ and
\[
O ( \frac1N) =  \frac1N \sum_i  \left[ \sum_{ j \neq i } f(U^N_j (t)) [ \phi ( U^N_i (t) + \frac1N ) - \phi ( U^N_i (t)) - \frac 1N \phi'(U^N_i (t)) \right] .
\]

Since  $\phi $, $\phi ' $ and $\phi '' $ are bounded as well as $f$
(thanks to Assumption \ref{ass:2}) there is a constant $c$ so that
\[
|\gamma_1^N (t)| \le c\Big(1+ \langle U^N(t),id\rangle+ |\langle U^N(t),\psi\rangle|\Big)
\le c' \Big( 1+ \frac 1N \sum_i U^N_i(t)^2\Big) .
\]
By Proposition \ref{thm1.1bis},
$\dis{ \sup_{t \le T } E [ \gamma_1^N (t)^2 ]  \le c}$
for a constant $c$ not depending on $N.$

The proof of \eqref{conditions} for $\gamma_2^N (t)$ is simpler. We write
$L= L_{\rm fire}+L_{\la}$, where $L_{\rm fire}\phi$ and $L_{\la}\phi$ are
given by the first, respectively second, term on the right hand side of
\eqref{eq:generator0}.  Since $L_{\la}$ acts as a derivative we have
\[
L_\la
 \langle U^N(t), \phi\rangle ^2 - 2  \langle U^N(t), \phi\rangle  L_\la \langle U^N(t), \phi\rangle = 0
\]
as can be easily checked.  We have
\begin{multline*}
  \frac{1}{N^2} \sum_{ i ,  j } L_{\rm fire} ( \phi ( U^N_i (t)) \phi ( U^N_j (t))) = \\
= \frac{1}{N^2} \sum_{ i \neq  j } \Bigg[ \sum_{k \neq i, j }  f( U^N_k (t))[  \phi ( U^N_i (t ) + \frac1N) \phi ( U^N_j (t) + \frac1N ) -\phi ( U^N_i (t ) ) \phi ( U^N_j (t)  ) ]  \\
 + f ( U^N_i (t) ) [\phi ( 0) \phi ( U^N_j(t) + \frac1N) - \phi ( U^N_i(t) ) \phi ( U^N_j (t) ] \\
+ f ( U^N_j (t) ) [\phi ( 0) \phi ( U^N_i(t) + \frac1N) - \phi ( U^N_i(t) ) \phi ( U^N_j(t) ]\Bigg]  \\
+ \frac{2}{N^2 } \sum_i \Bigg[ \sum_{ k \neq i } f (U^N_k (t )) [ \phi^2 ( U^N_i (t) +\frac1N )   - \phi^2 (U^N_i (t))  ] +  f ( U^N_i (t )) [ \phi^2 ( 0 )  - \phi^2 ( U^N_i (t ) )  ] \Bigg] .
\end{multline*}
The same arguments used earlier show that the $L^2$-norm of this term is bounded uniformly in $t\in [0,T]$ and in $N$.
The $L^2$-norm of $ - 2  \langle U^N(t), \phi\rangle  L_{\rm fire} \langle U^N(t), \phi\rangle$ is also bounded
uniformly because $|\langle U^N(t), \phi\rangle| \le c$ and
we have already proved the bound for  $L_{\rm fire} \langle U^N(t), \phi\rangle$. We have thus proved \eqref{conditions} and finished the proof.  Observe that taking into account the signs
we could prove that  $ \gamma_2^N ( t) \to 0 $ as $ N \to \infty .$
\end{proof}

\section{Coupling the true with an auxiliary process}\label{sec:ydelta}

The natural step after having proved tightness is to prove propagation of chaos. This is however
not so simple in our model because the firing of a neuron (i.e.\ when its membrane potential jumps) affects {\it simultaneously the state} of the other neurons and not just their jumping rates, as usual in mean field models. For this reason
we follow a different strategy here. In order to overcome this difficulty, we introduce an auxiliary process which is from one side
a good approximation of the true one in the $ N \to \infty $ limit, and which, from the other side, is easy to handle in the same limit. The auxiliary
process is defined in the present section where we prove that it is close to
the true process uniformly in $N$.  In Section \ref{sec:hy} we study the
hydrodynamic limit for the approximating process. Section \ref{sec:6} will then conclude the proof of Theorems \ref{theo:main} and \ref{theo:main2}.

\subsection{The auxiliary process}
\label{subsec:auxiliary}
We work under Assumption \ref{ass:2} throughout the whole section. We fix a time mesh $\delta > 0 $ and approximate the process $U^N ( t)  $ for fixed $N$ by a process which is constant on time intervals $ [ n \delta, (n+1) \delta [ , $ $n \geq 0.$ Since $N$ is fixed we shall drop the superscript $N$ from $U^N(t)$ unless ambiguities may arise.

The auxiliary process is denoted by
$ Y^{(\delta)} ( n \delta ) $ and is defined at discrete times
$n\delta$, $n\in \mathbb N ,$ such that $( Y^{(\delta)} ( n \delta ))_{n\in \mathbb N}$ is a Markov chain.
Its transition probability describes a process where neurons fire with constant firing rate
$f(y_i)$ in the time interval  $[n\delta,(n+1)\delta[ .$ Moreover, all firing events after the first one are
suppressed. Finally, the new configuration of neurons at time $ (n +1) \delta$ is obtained by first letting
the neurons evolve (for a time $\delta$) under the action of the gap-junction interaction
and then taking into account the effect of the firings at the end of the time interval. The precise definition is given now.

We put $ Y^{(\delta)} ( 0)  = U ( 0) $ and then proceed by induction on $n.$ Conditionally on $ Y^{(\delta)} ( n \delta ) = y = (y_1, \ldots , y_N),$ we choose $N$ independent exponential random variables $ \tau_1, \ldots , \tau_N , $ which are independent of anything else, having intensities $ f( y_i ) , i = 1, \ldots , N , $ respectively. We put
 \begin{equation}
 \label{eq:424}
\Phi_i (n  ) = 1_{ \{ \tau_i \le \delta \} }, 1 \le i \le N , \,  \, q = \frac1N \sum_{ i=1}^N \Phi_i (n ) ;
 \end{equation}
hence neurons $i$ such that $ \Phi_i ( n) = 1 $ spike during $ [n \delta , (n+1) \delta [,$ all other neurons do not spike during that interval. Notice that we keep constant the firing intensity of the neurons. We write
   \begin{equation}
   \label{eq:flow}
\varphi_{ \bar y , t}( y_i )  = e^{ - \lambda t } y_i + ( 1 - e^{-\lambda t }) \bar y , \; 0 \le t \le \delta ,
\;\; \bar y = \frac{1}{N} \sum_{ i=1}^N y_i
   \end{equation}
for the deterministic flow attracting position $ y_i$ to $ \bar y $ and set
\begin{equation}\label{319}
Y_i^{ (\delta ) } ((n+1) \delta ) = \varphi_{ \bar y , \delta } ( y_i ) + q  , \;   \mbox{ for all $ i $ such that $ \Phi_i(n) = 0 .$}
\end{equation}

Thus neurons which do not fire follow the deterministic flow. Moreover, we suppose that they feel the additional potential $ q ,$ generated by spiking of other neurons, only at the end of the interval $[ n \delta , (n+1 ) \delta [ . $

Let us now describe the evolution of the $ Nq $ neurons that fire. Let $ i_1 , \ldots , i_{Nq } $ be the labels of neurons such that $ \Phi_{i_j} (n)  = 1 , j = 1 , \ldots , Nq, $ ordered in such a way that $ \tau_{ i_j } >  \tau_{i_{j+1}} .$
We then assign the position
   \begin{equation}
   \label{319aa}
Y_{i_{Nq}}^{(\delta )} ((n+1) \delta ) =  \varphi_{ \bar y, \delta } ( 0 ) + (q - \frac1N )  = ( 1 - e^{ - \lambda \delta} ) \bar y + (q - \frac1N )
    \end{equation}
to the first neuron which has fired. This is the position of a neuron starting from potential $0$ at time $n \delta ,$ evolving according to the flow and receiving an additional potential $ q - \frac1N$ at time $(n+1) \delta ,$ due to the influence of the other spiking neurons (whose number is $Nq-1$).

The remaining $Nq - 1 $ neurons that spike are distributed uniformly in the following manner. We put
\begin{equation}\label{eq:dn}
d_n = \frac{\varphi_{ \bar y , \delta } ( 0) + ( q - \frac1N)}{Nq - 1} , \mbox{ if } Nq - 1 > 0 , \; d_n = \varphi_{ \bar y , \delta } ( 0) , \mbox{ if } Nq - 1 = 0 ,
\end{equation}
and
\begin{equation}\label{318}
 Y_{i_j}^{(\delta )} ((n+1) \delta )  =   (j-1)  d_n    ,  j = 1 , \ldots , Nq - 1 .
\end{equation}

\begin{rem}
The definition of the auxiliary process $Y^{(\delta )}$ has to be such that $Y^{(\delta)} $ is $\delta-$close to the original process. Therefore, we have some freedom in choosing the distribution of the spiking neurons in the auxiliary process and the above definitions (\ref{eq:dn}) and (\ref{318}) could be changed. However, the above choice is convenient for our purpose; we will see later that this precise choice enables us to produce strong convergence of the associated empirical measures to the limit equation, see also Remark \ref{rem:explanation} below.
\end{rem}

The analogue of Proposition \ref{thm1.1bis} holds
for the auxiliary process as well and it is straightforward to see that

\begin{prop}\label{lem:use}
 The variables $\sum_{ i=1}^N \Phi_i (n )$ are stochastically bounded
 by Poisson variables of intensity $Nf^*\delta$,    $f^* := \| f\|_\infty$.
\end{prop}

As a consequence, proceeding as in the proof of Proposition \ref{thm1.1bis},
for any $T$ there is $C$ so that for any initial datum  $x$ with $Y^{(\delta)} (0)=x$,
     \begin{equation}
     \label{eq:532}
P_x\Big[ \sup_{ n:n\delta \le T } \| Y^{(\delta)} (n\delta)\| \le \| x\| + 2f^* T \Big] \le e^{-C N T} .
     \end{equation}

\subsection{Coupling the auxiliary and the true process}
In order to show that $Y^{(\delta )} $ is close to the original process, we
couple the two Markov chains $ (U( n \delta ))_{n \geq 0 } $ and $ (
Y^{(\delta )} ( n \delta ))_{n \geq 0 }$ in such a way that neurons in both processes spike together as often as possible and such that the pair
$(U(n\delta), Y^{(\delta )} (n\delta))$, $n\in \mathbb N$, is a Markov
chain taking values in
$ \R_+^{ N \times N } .$

We start with $ Y^{(\delta )} ( 0 ) = U(0) .$ For any $n=0, 1,\dots$, given $(U(n\delta), Y^{(\delta )} (n\delta)),$ the values of $(U((n+1)\delta), Y^{(\delta )} ((n+1)\delta))$ will be chosen according to the
simulation algorithm given below. The algorithm uses the following  variables.
\begin{itemize}

\item $(x,y) \in \R_{+}^{ N} \times \R_{+}^{ N}$ and
$\bar x = \frac{1}{N} \sum_{i=1}^N x_i $. The strings $x$ and $y$ represent the state of the
neurons in the two processes and $\bar x$ gives the average potential of $x$.

\item Independent random times  $\tau_i^1  \in (0, +\infty)$, $\tau_i^2  \in (0,
  +\infty) $ and  $\tau_i  \in (0, +\infty) $, for all $i=1,\ldots, N$. These variables will
  determine the times of possible updates.

\item $m=(m_1,\ldots, m_N) \in \{0,1\}^N$. The variable $m_i$
  indicates the occurrence of a spike for neuron $i$ in the auxiliary
  process.

\item $ K \in \{ 0 , \ldots , N \} .$ The variable $K$ counts the number of spikes in the auxiliary process.

\item $ j = (j_1, \ldots , j_N ) \in \{ 0, 1 , \ldots , N \}^N .$ The variable $j_i $ is the label of the neuron associated with the $i-$th occurrence of a spike in the auxiliary process.

\item $L \in [0,\delta]$. The variable $L$ indicates the remaining time after every
  update of the variables. The simulation algorithm stops when $L=0$.

\end{itemize}

The deterministic flow attracting position $x_i$ to the average potential
$ \bar x$, given in \eqref{eq:flow}, will appear in the
algorithm. For convenience of the reader we recall its definition here
\[
\varphi_{ \bar x , t}( x_i )  = e^{ - \lambda t } x_i + ( 1 - e^{-\lambda t }) \bar x , \; 0 \le t \le \delta ,
\;\; \bar x = \frac{1}{N} \sum_{ i=1}^N x_i\, .
\]

Before proceeding further, let us explain the coupling. Given $ U(n \delta ) = x = (x_1, \ldots , x_N) $ and $ Y^{(\delta )} (n \delta ) = y = (y_1, \ldots , y_N),$ we start by associating to each neuron $i$ two independent stopping times $\tau_i^1 $ and $ \tau_i^2 .$ Here, $\tau_i^1 $ has intensity $  f(  \varphi_{ \bar x , t } ( x_i)) \wedge f (y_i )$ and $ \tau_i^2 $ is of intensity $|f(  \varphi_{ \bar x , t } ( x_i)) -f (y_i )|.$  Stopping times associated to different neurons are independent. If $\tau_i^1 $ rings first, then $U_i $ and $Y^{(\delta)}_i$ spike together, and the coupling is successful. However, if $\tau_i^2 $ rings first, then either $U_i $ spikes and $Y_i^{(\delta )} $ does not (this happens if $ U_i > Y_i^{(\delta)},$ details are given in lines $17-22$ of the algorithm) or vice versa. Once neuron $i$ has spiked in the auxiliary process we set $m_i = 1 $ and do not consider any spikes for neuron $i$ in the auxiliary process any more. Therefore, the next time to be considered for neuron $i$ is simply the next spiking time in the original process which is of intensity $f( \varphi_{ \bar x , t } ( x_i)).$  This time is called $\tau_i $ in our algorithm. All stopping times are only taken into account if they appear during the time interval $[0, \delta ] $ that we consider.

Our algorithm is given below. In the remainder of this section we shall prove that this is indeed a good coupling of the two processes.

\begin{algorithm}\label{thecoupling}
\caption{Coupling algorithm}
\begin{algorithmic}[1]
\STATE {{\it Input:} $(U(n\delta), Y^{(\delta )} (n\delta)) \in \R_{+}^{ N} \times \R_{+}^{ N}$}
\STATE {{\it Output:} $(U((n+1)\delta), Y^{(\delta )} ((n+1)\delta)) \in \R_{+}^{ N} \times \R_{+}^{ N}$}
\STATE {{\it Initial values:}  $(x,y)\gets (U(n\delta), Y^{(\delta )}
  (n\delta)),\, K \gets 0 \, , L \gets \delta\, ,m_i \gets 0 ,$ \mbox {for all}
$i=1,\ldots, N$, $ j_i \gets 0 , $ \mbox{ for all } $ i = 1 , \ldots , N $}
\WHILE {$ L > 0$}
\STATE {For $i=1,\ldots,N$, choose independent  random times}
\begin{itemize}
\item $\tau_i^1  \in (0, +\infty) $
with intensities $ f(  \varphi_{ \bar x , t } ( x_i)) \wedge f (y_i )$
\item $\tau_i^2  \in (0, +\infty) $
with intensities $|f(  \varphi_{ \bar x , t } ( x_i)) -f (y_i )|$
\item $\tau_i \in (0, +\infty) $ with intensities $ f( \varphi_{ \bar
    x , t } ( x_i))$ \STATE $R= \inf\limits_{ 1 \le i \le N ; m_i = 0 }  ( \tau_i^{1}  \wedge  \tau_i^{2 })  \wedge  \inf \limits_{ 1 \le i \le N ; m_i = 1 }  \tau_i  .$
\end{itemize}
\IF {$R \ge L$}
\STATE {{\it Stop situation:} \\
\STATE $ x_i \gets  \varphi_{ \bar x , L  } ( x_i ) $ for all $ i = 1 , \ldots , N $\\
\STATE $L \gets 0$ \\
\STATE $ y_i \gets  \varphi_{ \bar y , \delta  } ( y_i )  + \frac{K}{N} $ for all $ i = 1 , \ldots , N  $ such that $ m_i = 0 $\\
\STATE $  y_{j_1} \gets \varphi_{ \bar y , \delta } ( 0 ) + \frac{K- 1}{ N },$ if $ K \geq 1 $   \\
\STATE $  y_{ j_k } \gets (K - k ) \left[ \frac{\varphi_{ \bar y , \delta } ( 0 ) + (K- 1)/ N }{K- 1 } 1_{ \{ K > 1 \} }  \right] $ for all $ k = 2 , \ldots , K $ }
\ELSIF {$R=\tau_i^{1} < L\; $}
\STATE $m_i \gets 1,\; \; K \gets K + 1 , \; \; j_K \gets i , \; \;  L\gets (L-R)$
\STATE $x_i \gets 0 \;\; \mbox{and} \;\; x_j \gets  \varphi_{ \bar x , R } ( x_j)  + \frac1N, \; \mbox{for
all} \; \;  j \neq i$
\ELSIF {$R = \tau_i^{  2 } < L \; $ }
       \IF {$f(y_i ) > f ( \varphi_{ \bar x , R }(x_i )) \; $ }
       \STATE $m_i \gets 1 , $ $K \gets K + 1 , $ $j_K \gets i , $ $ L \gets ( L- R ) , $ $x_j \gets \varphi_{\bar x , R } ( x_j ) $ for all $j$
       \ELSIF {$f(y_i ) \le  f ( \varphi_{ \bar x , R }(x_i )) \; $ }
       \STATE $ L\gets (L-R),$ $ x_i \gets 0 $ and $ x_j \gets \varphi_{ \bar x , R } ( x_j)  + \frac1N, \; \mbox{for
all} \; \;  j \neq i$
       \ENDIF
\ELSIF {$R=\tau_i < L\; $}
\STATE $  L\gets (L-R)$
\STATE $x_i \gets 0 \;\; \mbox{and} \;\; x_j \gets  \varphi_{ \bar x , R } ( x_j)  + \frac1N, \; \mbox{for
all} \; \;  j \neq i$
\ENDIF

\ENDWHILE
\STATE $(U((n+1)\delta), Y^{(\delta )} ((n+1)\delta)) \leftarrow (x,y)$.
\RETURN $(U((n+1)\delta), Y^{(\delta )} ((n+1)\delta))$.
\end{algorithmic}
\end{algorithm}

\subsection{Closeness between the auxiliary and the true process}

The main result in this section, Theorem \ref{prop:coupling} below,  states
that  the auxiliary and the true processes are close to each other.
This means that
for most neurons, the potentials in the two processes are close to each other
(proportionally to $\delta$), while the
remaining ones constitute a small fraction of the totality
(also proportional to $\delta$).

\begin{defin}
\label{goodlabels}
A label $i \in \{ 1, \ldots , N \}$ is called ``good at time $k \delta$'' if for all $n=1 , \ldots , k$ the following is true.

\begin{itemize}
\item[]
Either $\Phi_i(n-1)=0$ and $U_i$ has not fired during the whole time interval $[(n-1)\delta,n\delta].$
\item[]
Or $\Phi_i(n-1)=1$ and $U_i$ has
fired exactly once in the time interval $[(n-1)\delta,n\delta]$.
\end{itemize}
We call $\mathcal G_n$ the set of good
labels at time $n\delta$ and $M_n = N - |\mathcal G_n| $ the cardinality of its complement.
If $i \in \mathcal G_k$  we call $D_i(k):= |U_i(k\delta)-Y^{(\delta)}_i(k\delta)|$. Finally, we set
$$ \theta_n= \max\{D_i(k), \, i \in  \mathcal G_n\, , k \le n \}.$$
\end{defin}

Then the following holds.
\begin{theo}\label{prop:coupling}
Under  Assumption \ref{ass:2}, for any fixed $ T>0,$
there exist $ \delta_0 > 0 $ and a constant $C$ depending on $f^* $
and on $ T$ such that for all $\delta \le \delta_0,$ with probability  $\geq 1 -e^{ - CN \delta^2 },$
             $$
\theta_n \le C \delta \quad \mbox{ and}
\quad \frac{M_n }{N} \le C \delta \quad \mbox{ for any fixed $n$ such that $ n \delta \le T .$ }
            $$
\end{theo}

{\em Strategy of proof.}
It is clear that $ M_{n- 1} \le M_n \le \ldots, $
because there is no recovery from not being a good label. We shall first prove
that till when $\theta_n \le c\delta$ the increments $M_n-  M_{n- 1} \le c'\delta^2 N$.
In fact  a label $i$   becomes bad at $n\delta$ if
in the time interval $((n-1)\delta,n\delta)$ there are either  two or more fires
of $ U_i(\cdot)$ (which cost $O(\delta^2)$) or else the clock $\tau_i^{2}$ (recall the algorithm given in Subsection
\ref{thecoupling})
rings, which also costs $O(\delta^2)$.  Since $n\delta \le T$ the sum
of the increments is then bounded by $c\delta$ as desired.  Thus there may be of order $c\delta N$ neurons which fire quite differently in the two processes but this
produces a change for the potential of the good labels of the order of
$\frac 1N (c\delta N) T$ which is also what is claimed in the theorem. The above heuristic argument can be made rigorous; the precise proof is given in the Appendix \ref{sec:proofcoupling}.

\subsection{Corollaries}

We conclude the section with a corollary of the above results which will be used in the analysis
of the hydrodynamic limit $N\to \infty.$ Recall that by considering the associated empirical measures (\ref{eq:emp}),
we interpret $U (t) $ and $Y^{(\delta)} (t) $ as elements of $\mathcal S ' .$

\begin{defin}\label{def:3}
We introduce the space $\F$ of  smooth functions $\phi(m)$,  $m\in \mathcal S ',$
which have the form
\begin{equation}
\label{eq:442}
 \phi(m) = h( m[a_1],...,m[a_k]), \;\; \text{$k$ a positive integer,}
\end{equation}
where $h(r_1,..,r_k)$ is a smooth function on $\mathbb R^k,$ uniformly Lipschitz continuous with Lipschitz constant $c_h,$ i.e.
\begin{equation}
\label{2.12}
|h(r_1,..,r_k)- h(r'_1,..,r'_k)|\le c_h \Big( \sum_{i=1}^k|r_i-r_i'|\Big) .
\end{equation}
The functions $a_i,i=1,..,k,$ in (\ref{eq:442}) are $C^\infty-$functions on $\R ,$ each one with compact support contained in
$\{|x| \le c\}$, $c>0. $
\end{defin}

Let $c'$ be an upper bound for the derivatives $|a'_i(r)|$, $i=1,..,k$. We also introduce
\begin{equation}
\label{2.13}
\mathcal T = \Big\{ t \in [0,T]: t = n2^{-k}T, k,n \in \mathbb N\Big\} .
\end{equation}
Recall that $P^{(N, \lambda)}_x$ denotes the law under which $U ( \cdot )$ starts from $ U(0) = x.$ Denote by $S^{(\delta,N, \lambda)}_x$ the law under which its approximation $Y^{(\delta)} ( \cdot)$ starts from $x$ at time $0,$ and write $Q^{(\delta,N, \lambda)}_x$ for the probability law governing the coupled process defined above. By abuse of notation, we shall also denote the associated expectations by $P^{(N, \lambda)}_x , S^{(\delta,N, \lambda)}_x$ and $Q^{(\delta,N, \lambda)}_x.$

\begin{prop}\label{prop2.1}
Let $t \in \mathcal T$, $\delta\in \{2^{-l} T, l \in \mathbb N\}$ such that $t=\delta n$ for some positive integer $n.$
Let $\phi$ as in \eqref{eq:442} with constants $c_h$, $c$ and $c'$.
Then, with $C$ as in Theorem \ref{prop:coupling} ($C$ is independent of $\delta$)
\begin{equation}\label{eq:445}
|P^{(N, \lambda)}_x [\phi (\mu_{ U(t)})]- S^{(\delta,N, \lambda)}_x [\phi (\mu_{ Y^{(\delta)}(t)})]|
\le k c_h c' \frac{C}{\delta^2 }e^{-C\delta^2 N} c + \delta(2 k c_h c' C) .
\end{equation}
\end{prop}

\begin{proof}
The left hand side of \eqref{eq:445} is not changed
if we replace $U(t)$ and $Y^{(\delta)} (t) $ by  $U^* (t) $ and $Y^{ (\delta) , *} (t) $ which are defined by setting
\[
U_i^* (t)= \min\{U_i (t),c\},\quad Y_i^{ (\delta ), *} (t)= \min\{Y^{(\delta)} _i(t),c\} ,
\]
$c$ as in Definition \ref{def:3}.
Let $\phi$ be as in \eqref{eq:442}, then by \eqref{2.12}
\begin{equation*}
|\phi (\mu_{U(t)}) -  \phi (\mu_{ Y^{(\delta)} (t)})|= |\phi (\mu_{U^*(t)}) -  \phi (\mu_{ Y^{(\delta), *} (t)})| \le  k c_h c' \frac 1N\sum_{i=1}^N |
U_i^* (t)-Y_i^{( \delta) , *} (t)| .
\end{equation*}
Hence
\begin{equation*}
|P^{(N, \lambda)}_x [\phi (\mu_{ U(t)})]- S^{(\delta,N, \lambda)}_x [\phi (\mu_{ Y^{(\delta)}(t)})]| \le
k c_h c' Q^{(\delta,N, \lambda)}_x \Big[ \frac 1N\sum_{i=1}^N |
U_i^* (t)-Y_i^{ (\delta ) , *} (t)|\Big] .
\end{equation*}
Let $ e^{-C\delta^2 N}$ be the bound on the bad events in the estimates of
the coupled process, obtained in Theorem \ref{prop:coupling}. Then
\begin{equation}\label{2.17}
  Q^{(\delta,N, \lambda)}_x \Big[ \frac 1N\sum_{i=1}^N |
U_i^* (t)-Y_i^{ (\delta ), *} (t)|\Big] \le \frac
C{\delta^2}  e^{-C\delta^2 N}c + 2C \delta
\end{equation}
where we used that $| U_i^* (t)-Y_i^{( \delta ), *} (t)| \le c$.
\end{proof}

\section{Hydrodynamic limit for the auxiliary process}\label{sec:hy}
The main result of this section is given in Theorem \ref{theo:2} below.
It states that the auxiliary process converges in the hydrodynamic limit
to the evolution defined in Subsection \ref{sec:limittraj}. When necessary we shall make  explicit the
dependence on $N$ writing $Y^{(\delta)} = Y^{(\delta,N)}.$
We then suppose that for all $\delta  \in \{ 2^{ - l} T, l \in \N \}  , $
$ Y^{(\delta,N)} (0 ) = x^{N},$ where $x^N$ satisfies
Assumption \ref{ass:0} of Section \ref{sec:def}. We will then show that the law of $ \mu_{ Y^{ (\delta , N) } } $
converges weakly to a process supported by a  single trajectory.

\subsection{The limit trajectory of the auxiliary process}\label{sec:limittraj}

In this subsection we describe the limit law of
$ \mu_{ Y^{ (\delta , N) } } $ denoted by $\rho^{(\delta)}_{\delta n}(r)$.
We start with some heuristic considerations which will motivate the expression
which defines $\rho^{(\delta)}_{\delta n}(r)$
and which foresee the way we shall prove   convergence to $\rho^{(\delta)}_{\delta n}(r)$.

{\bf Heuristics.} Consider an interval $I=[a,b]\subset
\mathbb R_+$ of length $\ell$ and center $r$.  We choose $\ell = N^{-\alpha}$,
$\alpha >0$ and properly small. The density of the initial configuration $x^{N}$
in $I$ is the average $\mu_{x^{N}}(I)$; our Assumption \ref{ass:0} ensures that
$\mu_{x^{N}}(I) = \frac{|x^{N}\cap I|}{N} \approx \psi_0(r) |I|$.  At time $\delta$
the neurons initially in $I$ and which do not fire
will be in the interval $J=[a',b']$ having center denoted by $r'.$ Here, recalling
\eqref{eq:424} and \eqref{eq:flow} for notation,
   \[
   a' = \varphi_{ \bar x^N , \delta }(a) + q^N,      \quad
   b' = \varphi_{ \bar x^N , \delta }(b) + q^N ,
   \]
where $q^N = q $ is the proportion of neurons that have fired, see \eqref{eq:424}.
By the definition of     $\varphi_{ \bar x^N , \delta }$,
$|J|= b'-a'= e^{-\la\delta} |I|$.  The only
neurons in   $J$ at time $\delta$ are those initially
in $I$ which do not fire, hence their number is approximately $|x^{N}\cap I|e^{-f(r)\delta}$.  Thus
    \[
\rho^{(\delta)}_\delta(r') |J| \approx \mu_{Y^{(\delta,N)}_\delta} (J)  \approx  e^{-f(r)\delta} \psi_0(r) |I|,\quad
\rho^{(\delta)}_\delta(r') \approx e^{\la\delta}e^{-f(r)\delta} \psi_0(r) ,
    \]
which gives  $\rho^{(\delta)}_\delta(r')$ in terms of $\rho_0(r)= \psi_0 (r),$ once we consider $r=r(r')$ which is given by
    \[
 r= \varphi_{ \bar x^{N} , \delta }^{-1}(r'-    q^N) \approx \varphi_{\bar \psi_0,\delta}^{-1}(r'- p^{(\delta)}_0\delta)
 = \varphi_{\bar \psi_0,\delta}^{-1}(r')-e^{\la\delta} p^{(\delta)}_0\delta ,
    \]
where $\bar \psi_0 = \int x \psi_0(x) dx$, $p^{(\delta)}_0=
\int \psi_0(x)  \frac {1 - e^{-\delta f(x)}}\delta dx$ are obtained by letting $N\to \infty .$ The inverse of
$ \varphi_{ \bar x , \delta } ( \cdot ) $, see \eqref{eq:flow}, is
        \begin{equation}
        \label{eq:inverse}
\varphi_{ \bar x , \delta }^{ - 1} ( x) = e^{ \lambda \delta }
\Big( x - (1 - e^{ - \lambda \delta }) \bar x\Big) .
        \end{equation}
The above gives a formula for $\rho^{(\delta)}_\delta(r')$ for all
\[
r'\ge r'_0=\varphi_{ \bar x^N , \delta }(0) + q^N =  ( 1 - e^{-\lambda \delta })  \bar x^N +q^N \approx ( 1 - e^{-\lambda \delta })\bar \psi_0 + p^{(\delta)}_0\delta  ;
\]
$r'_0$ is the same as in \eqref{319aa}. The definition of $Y^{(\delta,N)}_\delta$
is such that all the neurons which have fired are put uniformly in $[0,r'_0]$, thus
   \[
   \rho^{(\delta)}_\delta(r') \approx \frac{ p^{(\delta)}_0\delta}{p^{(\delta)}_0\delta+  ( 1 - e^{-\lambda \delta })\bar \rho_0  },\quad r' \le r'_0 .
   \]

{\bf Definition of the limit trajectory.} The definition of  $\rho^{(\delta)}_{n \delta }(r)$ will extend
and formalize the above definitions to all $n\delta \le T$.
We put $ \rho^{(\delta)}_0 (x)  = \psi_0 (x) ,$ where $\psi_0 $ is
a smooth probability density on $\R_+$ satisfying Assumption \ref{ass:0}.
We then proceed inductively in $ n $ such that $ n \delta \le T.$
Suppose that $\rho^{(\delta)}_{ n \delta } $ has already been defined. Then we put
\begin{equation}
\label{eq:549}
p^{(\delta)}_{n\delta} : = \int_0^\infty  \rho^{(\delta)}_{n\delta}(x) \frac{1-e^{-\delta f(x)}}{\delta} dx ,
\end{equation}
\begin{equation}\label{eq:undelta}
{\bar \rho}^{(\delta )}_{ n \delta } :=    \int_0^\infty  x  \rho^{(\delta)}_{n\delta}(x) dx ,
\end{equation}
and we define for all
$
x \geq r_n = ( 1 - e^{ - \la \delta } ) {\bar \rho}^{(\delta )}_{ n \delta }  + p^{(\delta)}_{n\delta}  \delta ,
$
\begin{equation}
\label{eq:550}
\rho^{(\delta)}_{(n+1)\delta}(x) =e^{ \lambda \delta } \; \rho^{(\delta) }_{n\delta}\left( \varphi_{ {\bar \rho}^{(\delta )}_{ n \delta } , \delta }^{ - 1} ( x)  -  e^{ \lambda \delta }  p^{(\delta)}_{n\delta}\delta  \right)
e^{  - f \left(\varphi_{{\bar \rho}^{(\delta )}_{ n \delta } , \delta }^{ - 1} ( x)  -  e^{ \lambda \delta }  p^{(\delta)}_{n\delta}\delta \right)\delta  }.
\end{equation}
Finally we put
      \begin{equation}
      \label{eq:conditionaubord}
\rho_{(n+1) \delta}(x)\equiv
\frac{ p^{(\delta)}_{n\delta} \delta }{ p^{(\delta)}_{n\delta}\delta  + (1 - e^{ - \lambda \delta } )  {\bar \rho}^{(\delta )}_{ n \delta } }
\quad  \mbox{ for all $x \in ]   - \infty , r_n [ $.}
      \end{equation}
In this way, $ \rho^{(\delta)}_{( n+1) \delta }  $ are probability densities on $\R_+$ for all $n,$ i.e.
\begin{equation}\label{eq:prob}
1 = \int_0^\infty \rho^{(\delta)}_{(n+1) \delta}(x) dx  .
\end{equation}

\begin{rem}\label{rem:explanation}
The fact that we have extended the definition of $ \rho_{(n+1) \delta}(x) $ to $ \R_-$ will be useful in the sequel. Notice that as $ \delta \to 0,$  (\ref{eq:conditionaubord}) reads as
$ \rho_{(n+1)\delta}(0)\sim \frac{  p^{(\delta)}_{n\delta}}{  p^{(\delta)}_{n\delta} + \lambda  {\bar \rho}^{(\delta )}_{ n \delta } } $
which corresponds to (\ref{eq:initial0}).
\end{rem}

Notice that if $ \rho^{(\delta)}_{ n \delta }$ has support $ ] - \infty , R_n ] , $ then the support of $\rho^{(\delta)}_{(n+1)\delta}$ is included in $] - \infty ,  R_{n+1} ],$ where
    $$
    R_{n+1} = e^{ - \lambda \delta } R_n +  p_{n \delta }^{(\delta )} \delta  + ( 1 - e^{ - \lambda \delta } ) {\bar\rho}^{(\delta )}_{ n  \delta }.
    $$
This leads to the following definition.

\begin{defin}[Edge]
We call $R_0$ the edge of the profile $ \rho_0 $ and
\begin{equation}
    \label{R_n}
    R_n = e^{ - \lambda \delta } R_{n-1} +  p_{(n-1) \delta }^{(\delta )} \delta  + ( 1 - e^{ - \lambda \delta } ) {\bar\rho}^{(\delta )}_{ (n-1) \delta }
    \end{equation}
the edge of $\rho^{(\delta )}_{n \delta } .$
\end{defin}
Noticing that
\begin{equation}
\label{3.4.1}
p^{(\delta)}_{n\delta} \le \int \rho^{(\delta)}_{n\delta}(x) \frac{1-e^{-\delta f^*}}{\delta} dx
=\frac{1-e^{-\delta f^*}}{\delta} \le f^*  \quad \mbox{ and } \quad \bar \rho_{(n-1) \delta }^{(\delta)} \le R_{n-1},
\end{equation}
it then follows that
\begin{equation}
\label{3.4.2}
R_{n} \le R_{n-1} +f^* \delta \le R_0+ f^* n\delta  \le R_0 + f^* T ,
\end{equation}
since $ n \delta \le T.$ Hence the supports of $\rho^{(\delta)}_{n\delta}$ are uniformly bounded. By iterating \eqref{eq:550} and by using the explicit form of the inverse flow $\varphi_{ \bar x , \delta}^{ - 1 } ( x) ,$ we get the explicit representation
\begin{multline}
\label{3.5}
\rho^{(\delta)}_{(n+1)\delta}(x) =e^{  \lambda (n+1) \delta } \; \psi_{0}\left(e^{  \lambda (n+1) \delta } x- (1 - e^{ - \lambda \delta}  )  \sum_{ k=0}^n e^{ \lambda (k+1) \delta }{\bar\rho}^{(\delta ) }_{k  \delta }   -  \sum_{k=0}^n e^{ \lambda (k+1) \delta } p^{(\delta)}_{k\delta}\delta\right) \\
 \exp\Big\{-\sum_{k=0}^n \delta  f\Big( e^{  \lambda (h+1 - k ) \delta }   x- (1 - e^{ - \lambda \delta }) \sum_{h=k}^n  e^{ \lambda (h+1-k) \delta } {\bar\rho}^{(\delta ) }_{k  \delta }
-\sum_{h=k}^n e^{ \lambda (h+1-k) \delta } p^{(\delta)}_{h\delta}\delta\Big) \Big\} ,
\end{multline}
for all
\begin{equation}\label{eq:xnstar}
x \geq x_{n+1}^* = e^{ - \lambda (n+1) \delta } \left( (1 - e^{ - \lambda \delta}  )  \sum_{ k=0}^n e^{ \lambda (k+1) \delta }{\bar\rho}^{(\delta ) }_{k  \delta }
+ \sum_{k=0}^n e^{ \lambda (k+1) \delta } p^{(\delta)}_{k\delta}\delta \right) ,
\end{equation}
where $ \psi_0 $ is the initial density. Notice that for all $ x > x_{n+1}^* , $ $\rho^{(\delta)}_{(n+1)\delta}(x) $ is continuous in $x. $ On $[0, x_{n+1}^* [, $ however, discontinuities are introduced. The following proposition shows that they are of order $ \delta .$

\begin{prop}\label{prop:5}
There exists a constant $C$ depending only on $ f^* , $ $\|f\|_{Lip} , $ $T$ and $R_0,$ such that
$$ | p_{n \delta }^{(\delta )} - p_{(n-1) \delta }^{(\delta )}| + | {\bar\rho}^{(\delta )}_{ n  \delta } - {\bar\rho}^{(\delta )}_{ (n-1)  \delta } | \le C \delta , $$
for all $ n $ such that $ n \delta \le T.$
\end{prop}

\begin{proof}
The proof is straightforward, using the Lipschitz property of $f$ and the fact that the supports of $\rho^{(\delta)}_{n\delta}$ are uniformly bounded.
\end{proof}

The following is a direct consequence of the definition of $\rho_{ n \delta }^{(\delta)} $ and of Proposition \ref{prop:5}.
\begin{cor}
There exists a constant $c$ such that for any $ \delta = 2^{ - k } T $ with $k$ large enough, for any $ n ,l $ with $ n \delta \le T , l \delta \le T, $
\begin{equation}\label{eq:fact1}
| \rho_{n \delta }^{(\delta)} (x) - \rho_{n \delta}^{(\delta ) } (y) | \le c (|x-y| \vee \delta ) \mbox{ for all $ x, y \in [ 0, x_n^* [ , $ }
\end{equation}
and
$$ | \rho_{n \delta }^{(\delta)} (x) - \rho_{n \delta}^{(\delta ) } (y) | \le c |x- y | \mbox{ for all $ x , y \in [x_n^* , \infty [ .$}$$
Moreover, for all $ n , l \geq 0,$
\begin{equation}\label{eq:fact2}
 | \rho_{n \delta }^{(\delta)} (x) - \rho_{l \delta}^{(\delta ) } (x) | \le c |n- l  |\delta \; \mbox{ for any fixed  $ x \in [x_n^* \vee x_l^* , \infty [ \;  \cup \; ] 0, x_n^* \wedge x_l^* [ .$}
\end{equation}
Finally, if $\psi_0 $ satisfies the additional assumption (\ref{eq:border}), then also
$$ | \rho_{n \delta }^{(\delta)} (x_n^* +) - \rho_{n \delta}^{(\delta ) } (x_n^* - ) | \le c \delta $$
and
$$   | \rho_{n \delta }^{(\delta)} (x) - \rho_{l \delta}^{(\delta ) } (x) | \le c |n- l  |\delta $$
for all $x > 0 .$
\end{cor}

Finally, the following result will also be used in the sequel.

\begin{prop}\label{prop:6}
There exists $\delta_0 > 0 $ and a constant $C$ depending on $ f^* , $ $\delta, \|f\|_{Lip} , $ $T$ and $R_0,$ such that
$$ \rho_{ n \delta }^{(\delta ) } ( 0 ) \geq C \psi_0 ( 0 ) , $$
for all $ n >0  $ such that $n \delta \le T,$ for all $ \delta \le \delta_0 .$
\end{prop}

\begin{proof}
In what follows, $C$ will be a constant that might change from line to line. Thanks to our assumptions imposed on the function $f,$
there exists a constant $C$ such that
\begin{equation}
 f ( u x) \geq C f (x) \mbox{ for all  } x \in [0, R_0 + f^* T ], u \in [ e^{ - \delta_0 \lambda }, 1 ] .
\end{equation}
Then, using (\ref{eq:549}) and (\ref{eq:550}) and the fact that $f $ is non decreasing, for all $\delta \le \delta_0, $
\begin{multline*}
 p^{(\delta)}_{ (n+1) \delta } \geq C \int_0^\infty f (x) \rho_{ (n+1)  \delta }^{(\delta )} ( x) dx \\
\geq C \int_0^\infty f\left( e^{- \lambda \delta } x + (1 - e^{ - \lambda \delta } ) \bar \rho^{(\delta)}_{n \delta } + p^{(\delta)}_{ n \delta } \delta  \right)
\rho_{ n \delta }^{(\delta )} (x) e^{ - f (x) \delta } dx\\
\geq C e^{ - f^* \delta } \int_0^\infty f( e^{ - \lambda \delta } x ) \rho_{ n \delta }^{(\delta )} (x)
\geq C e^{ - f^* \delta } \int_0^\infty f (x) \rho_{ n \delta }^{(\delta )} (x)  \geq  C p^{(\delta)}_{ n \delta } .
\end{multline*}
In particular,
$$  p^{(\delta)}_{ n \delta }\geq C p_0,$$
where $C$ depends on $\delta $ and where $ p_0 = \int_0^\infty f(x) \psi_0 (x) dx.$ On the other hand, Proposition \ref{prop:5} implies that
$$ p^{(\delta)}_{ n \delta } \le p_0 + C n \delta , \; \bar \rho^{(\delta)}_{n \delta }  \le \bar \psi_0 + C n \delta ,$$
for all $n$ with $ n \delta \le T, $ which implies that
$$ \rho_{ (n+1)  \delta }^{(\delta) } (0) \geq  \frac{p^{(\delta)}_{ n \delta } }{p^{(\delta)}_{ n \delta }  + \lambda \bar \rho^{(\delta)}_{n \delta } } \geq C \frac{p_0}{p_0 + \lambda \bar \psi_0 + C T } = C  \geq  C \psi_0 (0) .$$
\end{proof}

\subsection{Discretization of the membrane potentials}
Let $ (Y^{(\delta)} (n \delta ) )_{n \le T / \delta } $ be the auxiliary process defined in Subsection \ref{subsec:auxiliary}, starting from $x = x^{N} $ according to Assumption \ref{ass:0} such that $ \| x \| \le R_0 .$ Recalling the definition of $\Phi_i( n  ) $ in (\ref{eq:424})
we put
       \begin{equation}
       \label{eq:555bis}
q(n \delta) = \frac{ \sum_{i=1}^N \Phi_i(n) }{N} , \, V(n\delta) = \frac{q(n \delta ) }{\delta }  , \; \bar Y^{(\delta)} (n\delta)  = \frac{  \sum_{i=1}^N  Y_i^{(\delta)} ( n \delta )}{N}
       \end{equation}
and then define the sequence of random edges $R_0' = R_0 , $
   \begin{equation}
   \label{R'_n}
R_n' := e^{- \lambda \delta }R_{n-1}'  + V ((n-1) \delta ) \delta + ( 1 - e^{ - \lambda \delta}) \bar Y^{(\delta)}((n-1)\delta)  .
   \end{equation}

We will compare $Y^{(\delta)} (n \delta) $ and the limit $\rho^{(\delta)}_{n\delta} $ within small intervals, starting to explore the respective supports $ [0, R_n'] $ and $[0, R_n ] $ from the right border of the support (edge). Doing so, we are sure to compare  configurations of neurons in both process that correspond and that have evolved in the same fashion in the two processes, with high probability.

In order to do so, we introduce a mesh of $\R_+$ which depends on $N$ and on time, where times have the form $t=n\delta$, $t\le T$.
The meshes at different times will be related as in the heuristic considerations in the beginning of this section.

\begin{defin}[Membrane potential mesh]
Let $ 0 < \alpha \ll \frac{1}{6} .$ Given $N,$ let
$r \in [\frac 12,1]$ be such that $R_0$ is an integer multiple of $r N^{-\alpha}$.  We then partition $(-\infty,R_0]$ into intervals
      \begin{equation}
      \label{A.4}
 \mathcal I_{0}=\{I_{{i,0}},i\ge 1\}, \;\; I_{i,0} = ]R_0- i\ell,R_0-  (i-1)\ell], \quad
 \ell=rN^{-\alpha} ,
     \end{equation}
and define $\mathcal I'_{0}=\{I'_{{i,0}},i\ge 1\}$ by setting $I'_{i,0}=I_{i,0}$ so that at time $0,$
$\mathcal I'_{0} = \mathcal I_{0}$.  At times $n\delta$ we   define
$\mathcal I_{n\delta}=\{I_{{i,n}},i\ge i \}$ and
$\mathcal I'_{n\delta}=\{I'_{{i,n}}, i\ge 1\}$ as the sequences of intervals
\begin{equation}
\label{A.5}
I_{i,n}:=] R_n -  e^{ - \lambda \delta n } i \ell, R_n- e^{ - \lambda \delta n } (i-1)\ell] ,\quad
I'_{i,n}:=
] R'_n -  e^{ - \lambda \delta n }i \ell, R'_n- e^{ - \lambda \delta n } (i-1)\ell].
\end{equation}
     \end{defin}

The strategy is to compare the  ``mass'' of $\mu_{Y^{(\delta)}(n\delta)}$ in $I'_{i,n}$ and the
mass in the corresponding  interval $I_{i,n}$ (with same $i$)
for the limit $\rho^{(\delta)}_{n\delta}.$  We shall prove that for most intervals
the corresponding masses are close to each other in a sense to be made precise below. In order to do this properly, we
need to specify the mass distributions
in $\{x<0\}$ and to define ``bad'' intervals where the masses may differ. We start with the former.
We have already extended the density $\rho^{(\delta)}_{n\delta}(x)$ to $x<0$, see
\eqref{eq:conditionaubord}. For the neurons we proceed analogously
and extend  $\mu_{Y^{(\delta)}(n\delta)}$ to the negative axis by adding an infinite  mass
     \begin{equation}
     \label{eq:initial}
 \left( \mu_{ Y^{(\delta)} ( n \delta ) }\right)_{| \, ] - \infty , 0[} := \frac{1}{N} \sum_{ i=1}^\infty \delta_{- i  d_n    } , \quad  n \delta \le T
     \end{equation}
where in agreement with \eqref{eq:dn}
\begin{equation}
d_n = \frac{(1- e^{ - \lambda \delta })\bar Y^{(\delta)} (n\delta) + ( \delta V (n\delta)   - \frac1N)  }{N \delta V (n\delta)  - 1  } .
\end{equation}
Notice that the choice (\ref{eq:initial}) corresponds exactly to the initial configuration given in (\ref{318}). We introduce the following quantities for all $ {i, n } .$
\begin{equation}\label{A.6}
 N'_{i,n} =  N \mu_{ Y^{(\delta)} (n\delta)} ( I'_{i,n} ) ,\; N_{i,n} = N \int_{I_{i,n}} \rho^{(\delta)}_{n\delta}(x)\, dx, \; {\rm w}_i =\int_{I_{i,0}} \psi_0(x)\, dx ,
\end{equation}
where we extend the definition of $\psi_0 $ to $\R_-$ by putting $ \psi_0 ( x) = \psi_0 ( 0 ) $ for all $x < 0 .$ Notice that since $\psi_0 \geq c (x-R_0)^2$, $c>0$, in a left neighborhood of $R_0$,
\begin{equation}\label{A.8}
{\rm w}_i \geq c \ell^3 ,
\end{equation}
while, ``away'' from $R_0$, ${\rm w}_i \geq c\ell$, for some $c>0$. Finally we define the ``bad'' intervals as follows.

\begin{defin}[Bad intervals]
   \label{badintervals}
 $I_{i,n}$ is {\em bad,} if there is $n_0\le n$ such that (at least) one of the following four properties holds.
\begin{enumerate}
\item
$I_{i,n_0}\cap \{x<0\} \ne \emptyset$ and $I_{i,n_0}\cap \{x> 0\} \ne \emptyset.$
\item
$I'_{i,n_0}\cap \{x<0\} \ne \emptyset$ and $I'_{i,n_0}\cap \{x> 0\} \ne \emptyset.$
\item
$I_{i,n_0} \subset  \{x<0\} $ and $I'_{i,n_0} \subset \{x> 0\} .$
\item
$I'_{i,n_0} \subset  \{x<0\} $ and $I_{i,n_0} \subset \{x> 0\} .$
\end{enumerate}
$I'_{i,n}$ is  bad if $I_{i,n}$ is  bad. An interval is  {\em good} if it is not bad.
\end{defin}

\subsection{Hydrodynamic limit}
In order to compare $ Y^{(\delta)} ( n \delta ) , n \le \delta^{-1} T , $ and $\rho^{(\delta)}:=(\rho^{(\delta)}_{n\delta}, n\le \delta^{-1}T),$
we introduce the following distance.

\begin{defin}  [Distances]\label{distances}
We define for any $ n \le \delta^{-1}T,$
    \begin{equation}
    \label{numbbad}
B_n := \text{\rm number of bad intervals in $\mathcal I_n$}
    \end{equation}
and set
     \begin{multline}
     \label{A.9}
d_n ( Y^{(\delta)} ,\rho^{(\delta)}):=  B_n +
\max_{I_{i,n}\;{\rm good}, I_{ i, n } \subset \R_+} \frac{|N'_{i,n}- N_{i,n}| }{ {\rm w}_i N\ell}
\\
 + \frac \delta{ \ell }{|  V ((n-1)\delta) -   p_{(n-1) \delta}^{ (\delta) } | }+ \frac{| \bar Y^{(\delta)} (n \delta) - {\bar \rho}_{n \delta}^{ (\delta) } | }{ \ell }.
\end{multline}
\end{defin}

$d_n ( Y^{(\delta)} ,\rho^{(\delta)})$ depends on the times $\tau_j(k)$, $j=1,..,N$, $k \le n-1,$ where the
$\tau_j(k)$ are the times
which enter in the definition of $\Phi_j(k)$, see
\eqref{eq:424}. Let
$$\mathcal F_n = \sigma \{ \tau_j(k) , j = 1, \ldots , N ,  k \le n - 1 \}  $$
be the $\sigma$-algebra generated by these variables. Observe that $Y^{(\delta) }_{\delta n}$,
$\bar Y^{(\delta)} (n\delta)$ and $V ((n-1) \delta) $ are $\mathcal F_n$-measurable.
We prove in Theorem \ref{theo:2} below that with large probability (going to 1 as $N\to \infty$)
the distances $d_n( Y^{(\delta)} ,\rho^{(\delta)})$ are  bounded for all $n$ such that $ n\delta \le T$. Loosely speaking this is due to the fact that the auxiliary process is defined in terms of independent exponential random variables and that the initial configuration is made of i.i.d.\ random variables. The bounds on $d_n( Y^{(\delta)} ,\rho^{(\delta)})$
are given by coefficients $\kappa_n$ which do not depend on $N$ but have a very bad
dependence on $\delta$ for small $\delta$.  $\delta$ however is a fixed parameter in this section
and by the way $d_n$ is defined, the bounds imply that $ Y^{(\delta)}$ and $\rho^{(\delta)}$ become very
close in most of the space as $N\to \infty$ (and keeping $\delta$ fixed).

\begin{theo}\label{theo:2} Grant Assumptions \ref{ass:0} and \ref{ass:2}.
There exist  $\kappa_n > 0  , $ $ \gamma \in ] 0 , 1 [   ,$ a sequence $c_1(n) \in \R_+ $ which is increasing in $n$,
and a constant  $c_2>0$  such that
      \begin{equation}
      \label{A.10}
 S_{x}^{(\delta,N, \lambda )}\Big[d_n( Y^{(\delta)} ,\rho^{(\delta)}) \le \kappa_n\Big]  \geq 1 - c_1(n) e^{ - c_2 N^\gamma  },
       \end{equation}
for all $ n $ such that $ n \delta  \le T.$
            \end{theo}
The proof of Theorem \ref{theo:2} is given in the Appendix \ref{sec:prooftheo2}.

\begin{rem}\label{rem:relaxed}
We prove Theorem \ref{theo:2} under the strong Assumption \ref{ass:2} which can be weakened. Indeed, it is sufficient to impose \eqref{A.10} for $n = 0 .$ Recalling that by the definition of $\mathcal I_0$ all its intervals are good at time $ n = 0 , $ Assumption \ref{ass:2} clearly implies \eqref{A.10} for $n = 0 .$
\end{rem}

\begin{rem}
Theorem \ref{theo:2} gives strong convergence of $\mu_{ Y^{(\delta) } (t) }$ to $\rho^{(\delta)}_t (x) dx .$ Indeed, (\ref{A.10}) implies the convergence of the ``densities'' $ \frac{ N'_{i,n}}{ N \ell } $ (notice that $ {\rm w}_i \le \ell \to 0  $ as $ N \to \infty $).
    \end{rem}

As a corollary of Theorem \ref{theo:2} we obtain the desired convergence
\begin{cor}[Hydrodynamic limit for the approximating process]\label{cor:1}
Under the conditions of Theorem \ref{theo:2}, let $t \in \TT ,$ $ \delta \in \{ 2^{ - l T } , l \in \N \} $ such that $ t = \delta n $ for some positive integer $n.$ Then almost surely, as $ N \to \infty, $
$$ \mu_{ Y^{(\delta ) } (t) } \stackrel{w}{\to } \rho^{(\delta)}_t (x) dx .$$
\end{cor}

\section{Hydrodynamic limit for the true process}\label{sec:6}

We can now
conclude the proof of Theorem \ref{theo:main}. The convergence
in the hydrodynamic limit will be proved as a consequence of Proposition
\ref{prop:4} and of Corollary
\ref{cor:1}, which proves the convergence for the approximating
process, of (\ref{eq:fact1}), (\ref{eq:fact2}).

Given $T>0$ let $ \mathcal T:= \{ n \delta \le T: n\in \mathbb N,\; \delta = 2^{-k}T, k \in \mathbb N \} $.  Since $ \mathcal T$ is countable and $p^{(\delta)}_t$, $\bar \rho^{(\delta)}_t$ are bounded,
there exist bounded
functions $p^{(0)}_t$
and ${\bar\rho_t}^{(0)}$ on $ \mathcal T$ and a subsequence $(k_n)_n $ so that for all
$t\in \mathcal T$
 $$
p^{(0)}_t = \lim_{n \to \infty } p_t^{ ( 2^{ - k_n  }T ) },\quad
 {\bar\rho_t}^{(0)}= \lim_{n \to \infty } \bar\rho_t^{ ( 2^{ - k_n  } T ) }.
 $$
By Proposition \ref{prop:5} $p^{(0)}_t$ and ${\bar\rho_t}^{(0)}$ are continuous in
$\mathcal T$ and thus extend uniquely to continuous functions on $[0,T]$ which are denoted
by the same symbol; moreover, using again  Proposition \ref{prop:5} and
denoting below by $\delta$ elements of the form
$2^{ - k_n  } T$:
   \begin{equation}
   \label{77.1}
   \lim_{\delta\to 0} \sup_{n:n\delta \le T} \sup_{t\in [n\delta,(n+1)\delta)}\Big(
   |p_{n\delta}^{ ( \delta ) } - p^{(0)}_t| +  |\bar\rho_{n\delta}^{ ( \delta ) } - {\bar\rho_t}^{(0)}|\Big) = 0 .
        \end{equation}

 Define for any $t\in [0,T]$
 $$
  x_t^{ *, 0 } =
  e^{ - \lambda t} \left( \lambda \int_0^t e^{ \lambda s } {\bar \rho_s}^{(0)} ds + \int_0^t e^{ \lambda s } p_s^{(0)} ds \right) 
  $$
and, to underline the dependence on $\delta$,  rewrite the $x_n^* $
defined in \eqref{eq:xnstar} as  $x_{n\delta}^{*,\delta} $. Then, using \eqref{77.1},
   \begin{equation}
   \label{77.2}
   \lim_{\delta\to 0} \sup_{n:n\delta \le T} \sup_{t\in [n\delta,(n+1)\delta)}
   |x_{n\delta}^{*,\delta} - x_t^{ *, 0 }| = 0 .
        \end{equation}
Denoting below by $\eps$ elements in
$\{2^{ - k   }, k \in \mathbb N\}$, by \eqref{77.2} for any such $\eps$ there is $\delta_\eps$ so that for any $\delta < \delta_\eps$ the following holds. For all $t=n\delta \le T$ if
$|x- x_t^{ *, 0 }| \ge\eps$ then $x-x_{t}^{*,\delta}$ has the same sign as
$x- x_t^{ *, 0 }$.  We can then use  \eqref{eq:fact1} and \eqref{eq:fact2}
and a Ascoli-Arzel\`a type of argument to deduce that
$\rho^{(\delta)}_t (x) $ converges in sup norm by subsequences
to a continuous function $\rho_t (x) $, $t\in \mathcal T$, $|x- x_t^{ *, 0 }| \ge
\eps$ with compact support.  By continuity  $\rho_t (x) $ extends to all $t\in [0,T]$,  $|x- x_t^{ *, 0 }| \ge
\eps$.  By a diagonalization procedure we extend the above to all $x,t$ with
$t \in [0,T]$ and $x \ne x_t^{ *, 0 }$.  Then by  (\ref{eq:prob}), (\ref{eq:549}) and (\ref{eq:undelta}) for any $t\in \mathcal T$
  $$
  1 = \int_0^\infty \rho_t (x) dx , \quad p^{0}_t = \int_0^\infty \rho_t (x) f(x) dx
  $$
and
 $$
 {\bar\rho_t}^{(0)} =   \int_0^\infty  x \rho_t (x) dx ,
 $$
which, by continuity extend to all $t\in [0,T]$.  Thus $ p^{0}_t$ and
${\bar\rho_t}^{(0)}$ coincide with $ p_t$ and
${\bar\rho_t}$ given in \eqref {eq:217.1} and we shall hereafter drop the superscript 0.
Finally by taking the limit $\delta\to 0$ in \eqref{eq:550} and \eqref{eq:conditionaubord}
we prove that  $\rho_t(x)$ satisfies  \eqref{eq:ut-0}-- \eqref{eq:ut-1}.

We shall next prove that $\rho_t(x)$ is a weak solution of
\eqref{eq:217.3}--\eqref{eq:217.3.1} with $u_0=\psi_0$ and $u_1$ as in \eqref{eq:217.3.2}.

\begin{lem}
If $\rho_t(x)$ is given by
 \eqref{eq:ut-0}-- \eqref{eq:ut-1}
then for any test function $\phi$
   \begin{eqnarray}
   \label{77.3}
  \int_0^\infty \phi(x) \rho_t(x) dx &=& \int_0^\infty \psi_0(x) e^{-\int_0^t f(\varphi_{0,s}(x) ds)} \phi(\varphi_{0,t}(x)) dx \nonumber
  \\&+& \int_0^t  p_s  e^{-\int_s^t f(\varphi_{s,s'}(0) ds')}
   \phi(  \varphi_{s,t}(0))ds .
        \end{eqnarray}

\end{lem}

\begin{proof}
Calling $x^*_t=\varphi_{0,t}(0)$ we write
 \begin{equation}
  \label{77.4}
  \int_0^\infty \phi(x) \rho_t(x) dx =   \int_0^{x^*_t} \phi(x) \rho_t(x) dx
  +  \int_{x^*_t}^{\infty} \phi(x) \rho_t(x) dx .
  \end{equation}
In the second integral on the right hand side we make the change of variables $x \to y$ where $\varphi_{0,t}(y)=x$. Recalling \eqref{eq:flowlimit}, we have
 \[
  \frac{dx}{dy} = e^{-\la  t} .
  \]
Using that  $\rho_t(x)$ is given by
 \eqref{eq:ut-0} we can then check that the second integral on the right hand side of \eqref{77.4} is equal to
the first integral on the right hand side of \eqref{77.3}.

In the first integral on the right hand side of \eqref{77.4} we make
the change of  variables $x\to s$ where $\varphi_{s,t}(0)=x$. Using once more \eqref{eq:flowlimit}, we have
\[
  \frac{d\varphi_{s,t}(0)}{ds} = - V(0,\rho_s)e^{-\la(t-s)} .
  \]
Using this and recalling that  that  $\rho_t(x)$ is given by
  \eqref{eq:ut-1} we then complete the proof of the lemma.

\end{proof}

It follows from \eqref{77.3} that for any test function $\phi$,
$\int \phi \rho_t dx$ is differentiable in $t$ and that its derivative satisfies \eqref{eq:limitpde}.  Moreover by choosing $\phi= f$ and $\phi= x$ in \eqref{77.3} we then deduce that $p_t$ and $\bar\rho_t$
are differentiable and from this that $\rho_t(x)$ is differentiable in $t$ and $x$
in the open set $\mathbb R_+\times\mathbb R_+ \setminus\{(t,x): x=\varphi_{0,t}(0)\}.$ Hence by \eqref{eq:limitpde},  $\rho_t(x)$  satisfies  \eqref{eq:217.3}
in such a set and the boundary conditions \eqref{eq:217.3.1} with $u_0=\psi_0$ and $u_1$ as in \eqref{eq:217.3.2}.

\vskip.5cm
We shall next prove  uniqueness  for \eqref{eq:limitpde}. As a consequence the limit $\rho_t(x)$ we have found using compactness does not depend
on the converging subsequences, we therefore have
full convergence.
It is convenient to rewrite  \eqref{eq:limitpde} as follows.
For all $\phi \in C^1 ( \R_+, \R ),$ putting $ g(t, dx) = \rho_t (x) dx,$ 
\begin{multline}\label{eq:pde}
\partial_t \int \phi ( x)  g(t, dx)  = \int [ \phi (0 ) - \phi ( x) ] f(x) g(t,dx)  + \int \phi'  ( x) [ \lambda \bar \rho_t + p_t - \lambda x]  g(t, dx)  dx ,
\end{multline}
$ g(0, dx )  = \psi_0 (x)  dx, $ $ p_t = \int f(x) g(t ,dx) , $ $\bar \rho_t = \int xg(t,dx) .$

\begin{prop}
$\rho_t(x) dx $ is the unique solution of (\ref{eq:pde})
solving the initial condition $  \rho_0 (x)  = \psi_0 (x)$ for all $x $ and
$$1 = \int_0^\infty  \rho_t ( x) dx , \quad p_t = \int_0^\infty f (x) \rho_t ( x) dx , \quad \bar\rho_t = \int x \rho_t (x ) dx .$$
\end{prop}

\begin{proof}
We address the uniqueness of the solution. Any law $ g(t, dx) $ solving (\ref{eq:pde}) is the law of the Markov process $ U(t), t \geq 0,  $ which is solution of the non-linear SDE
\begin{equation}\label{eq:SDEnonlinear}
 d U(t) = ( - \lambda U(t) + \lambda E ( U(t) ) + E ( f(U(t))) dt  - U(t-) \int_{\R_+ } 1_{\{ z \le f( U(t- ) ) \}} N (dt, dz) .
\end{equation}
Here, $N(ds, dz) $ is a Poisson random measure on $\R_+ \times \R_+ $ having intensity $ds dz .$
It suffices to show the existence of a unique strong solution of the above non-linear SDE on a fixed time interval $ [0, T ].$

Let $U$ and $V$ be two solutions starting from $ U(0) = V(0) ,$ and write for short $ a_t = E ( \lambda  U(t) + f (U (t))) ,$ and $ a_t' $ for the corresponding quantity for $V.$

We start by giving a priori bounds on $ U(t) $ and $V(t).$ It follows directly from (\ref{eq:SDEnonlinear}) that
$$ E ( U(t)) \le E ( U(0)) + f^* t  \le R_0 + f^* T ,$$
for all $t \le T.$ But clearly
$$ U(t) \le U(0 ) + \int_0^t E ( \lambda U(s) + f(U(s) ) ) ds \le U(0) + f^* T + \lambda \int_0^t E ( U(s) ) ds \le C_T,$$
for all $ t \le T,$ where the constant $C_T$ depends only $U(0) ,$ $R_0,$ $f^* $ and $\lambda $ and where we recall that $R_0$ is the support of $ \psi_0 .$ The same upper bound holds obviously for $V(t) .$

Coupling $U$ and $V$ such that they have the most common jumps possible, we obtain
\begin{multline*}
\frac{d}{dt} E | U(t) - V(t) |  = \\
- E \left( f( U(t) ) \wedge f( V(t) ) | U(t) - V(t) | + | f( U(t) ) - f(V(t) ) | ( U(t) \wedge V(t) -  | U(t) - V(t) |  \right)\\
- \lambda E ( sign ( U(t) - V(t) ) ( U(t) - V(t) ) \\
+ \lambda E ( sign (U(t) - V(t) ) ( a_t - a_t' ) ) .
\end{multline*}
Since $f$ is non-decreasing, the first line is equal to
$$E  ( U(t) \wedge V(t) ) | f( U(t)) - f (V(t) ) | \le C C_T E  | (U(t) - V(t) |,$$
since $U(t) \wedge V(t) \le C_T.$  Moreover, it is evident that the second and third line are bounded from above by
$$ C  E  | U(t) - V(t) | .$$
Hence,
$$ \frac{d}{dt} E | U(t) - V(t) |  \le C  E | U(t) - V(t) | ,$$
for all $t \le T,$ implying that $ U(t) = V(t) $ almost surely, for all $t \le T .$

\end{proof}

\vskip.5cm

We shall now prove that the true process converges to $ \rho_t (x) dx $ in the hydrodynamic limit.
Call $ {\cal P}^N $ the law of the measure valued process $ \mu_{ U^N (t) } , t \in [0, T] .$ By the tightness proved in Proposition \ref{prop:4}, we have convergence by subsequences $ {\cal P}^{N_i} $ to a measure valued process $ {\cal P}.$ We will show that any such limit measure ${\cal P} $ is given by the Dirac measure supported by the single deterministic trajectory $ \rho_t (x) dx , t \in  [0, T], $ where $\rho_t(x) $ is the limit of $ \rho_t^{(\delta ) } (x) $ found above.

First of all we state the following support property.

\begin{prop}
Any weak limit $ {\cal P} $ of ${\cal P}^N $ satisfies
$$ {\cal P} ( C ( [0, T] , {\cal S}' ) ) = 1 ,$$
where $C ( [0, T ], {\cal S}' ) $ is the space of all continuous trajectories $[0, T ] \to  {\cal S}' .$
\end{prop}

\begin{proof}
The proof is analogous to the proof of Theorem 2.7.8 in De Masi and Presutti 1991.
\end{proof}

Let us denote the elements of $C ( [0, T ], {\cal S}' ) $ by $ \omega = ( \omega_t , t \in [0, T ] )$ and
let $t\in [0,T]$.
Suppose $ {\cal P} $ is the weak limit of $ {\cal P}^{N_i} $. We
shall prove that  $ {\cal P} $ is supported by
$\{\omega:\omega_t=\rho_t(x)dx\}$.  Thus $ {\cal P} $ coincides with $\rho_t(x)dx$ on the rationals of $[0,T]$ and
by continuity on all $t\in [0,T]$ and therefore any weak limit  of $ {\cal P}^N $ is supported
by $\rho_tdx$.

 The marginal  of ${\cal P}$ at time $t$ is determined by the
expectations
    $$
   \int  h\Big(\omega_t  (a_1 ),..,\omega_t(a_k)\Big) d {\cal P } =: {\cal P }_t(h)
    $$
where, as in  Definition \ref{def:3}, $h$ is a smooth
function on $\mathbb R^k$, $k\ge 1$, and $a_i$ are smooth
functions on $\mathbb R_+$ with compact support.  We need to show that
   \begin{equation}
   \label{needtoprove}
{\cal P }_t(h) = h\Big( \int  a_1 (x) \rho_t (x) dx , \ldots,\int a_k (x) \rho_t (x) dx \Big) .
        \end{equation}


In the sequel, $ t \in {\cal T}$ and $\delta \in \{2^{-n}T, n\ge 1\}$. For any $\varepsilon > 0$  there exists $ n_0 $ such that for all $ n\geq n_0 , $
     \begin{equation*}
| h\Big( \int a_1 \rho_t dx  , \ldots , \int a_k  \rho_t dx \Big) -  h\Big( \int a_1  \rho^{(2^{-n}T)}_t dx  ,\ldots ,\int  a_k  \rho^{(2^{-n}T)}_t dx \Big) | \le\varepsilon .
    \end{equation*}
Moreover there exists $N^*$ so that for all $N_i\ge N^*$,
\begin{equation*}
|  {\cal P}_t^{N_i }  (h) - {\cal P}_t (h) | \le \varepsilon ,
\end{equation*}
where $ {\cal P}_t^{N_i }  (h) := P_x^{(N_i, \lambda ) } ( h ( \mu_{U (t) } )),$ see \eqref{eq:445}.
By \eqref{eq:445} for $\delta$ small enough and  $N_i$ large enough
\begin{equation*}
| P_x^{(N_i, \lambda ) } (h( \mu_{U (t) } ))  - S_x^{(\delta , N_i, \lambda ) }  (h( \mu_{Y^{(\delta)} (t) } )) | \le   \varepsilon .
\end{equation*}
Applying Corollary \ref{cor:1} for $N_i$ large enough,
\begin{equation*}
|S_x^{(\delta , N_i, \lambda ) }  (h( \mu_{Y^{(\delta)} (t) } ))  -  h\Big( \int \rho^{(\delta)}_t a_1,..,\int\rho^{(\delta)}_t a_k\Big) | \le   \varepsilon .
\end{equation*}
Collecting the above estimates and by
the arbitrariness of $\varepsilon$ we then get \eqref{needtoprove}. This finishes the proof of Theorem \ref{theo:main}.   \hfill $\bullet $

\vskip.5cm
Finally, to prove  Theorem \ref{theo:main2} we need to show that
\begin{equation}
 \label{77.6}
\lim_{x\nearrow x^*_t} \rho_t(x)= \psi_0 \left(\varphi_{0, t }^{ - 1 } (x^*_t) \right)
\exp \left\{ - \int_0^t [ f - \lambda ] ( \varphi_{ s, t }^{ - 1 } (x^*_t) ) ds   \right\}
\end{equation}
where $x^*_t=\varphi_{0,t}(0)$.  For $x<x^*_t$ we use \eqref{eq:ut-1}
\begin{equation*}
 \rho_t (x) = \frac{p_s}{ p_s + \lambda \bar \rho_s  } \exp \left\{ - \int_s^t [f ( \varphi_{s, u } ( 0 ) ) - \lambda ] du \right\} 
\end{equation*}
with $s$ such that $\varphi_{s,t}(0) = x$.  By continuity
\[
 \lim_{s\to 0}\varphi_{s, u } ( 0 )=\varphi_{0, u } ( 0 )=\varphi_{ u, t }^{ - 1 } (x^*_t) .
\]
Moreover, since we have proved earlier the continuity of $p_s$ and $\bar \rho_s$
\[
\lim_{s\to 0}\frac{p_s}{ p_s + \lambda \bar \rho_s  } = \frac{p_0}{ p_0 + \lambda \bar \rho_0}
\]
so that \eqref{77.6} follows from \eqref{eq:border}.  \hfill $\bullet $

\appendix
\section{Proof of Theorem \ref{thm1.1}}\label{sec:proof1}
We are working at fixed $N$ and therefore drop the superscript $N$ from $U^N .$ Recall that the average potential of configuration $U(t) $ is given by
$
\bar U_N(t) = \frac 1N \sum_{i=1}^N U_i(t)
$
and let
\begin{equation}
\label{1.10}
K (t) = \sum_{ i = 1 }^N \int_0^t 1_{ \{ U_i  (s-) \le 2 \} } d N_i ( s)
\end{equation}
be the total number of fires in $[0,t]$ when $U_i\le 2 .$ Recall (\ref{eq:nt}). The key element of our proof is the following lemma which is due to discussions with Nicolas Fournier.

\begin{lem}\label{lem:nicolas}
We have
\begin{equation}
\label{1.11}
 \bar U_N(t) \le \bar U_N(0) + \frac{K(t)}{N} ,
\quad
 \frac{N(t)}N \le \bar U_N ( 0) + 2 \frac{K(t)}{N}
\end{equation}
and
\begin{equation}
\label{1.12}
 \|  U(t)\| \le 2 \| U (0)\| + 2 \frac{K(t)}{N} .
\end{equation}
\end{lem}

\begin{proof}
Suppose $U_i$ fires at time $t$, then
\[
\bar U_N(t)= \frac 1N \sum_{j\ne i} (U_j(t-) + \frac 1N)= \bar U_N(t-) + \frac{N-1}{N^2} - \frac{U_i(t-)}N .
\]
Thus the average potential decreases if $U_i(t-)\geq 1$ (and a fortiori if $ U_i ( t- ) \geq 2 $) which implies the first assertion of (\ref{1.11}).
Concerning the second assertion of (\ref{1.11}), we start with
$$ \bar U_N (t) = \bar U_N ( 0 ) + \frac1N \sum_{i=1}^N \int_0^t \left( \frac{N-1}{N} - U_i (s-) \right) d N_i (s) ,$$
which implies, since $ \bar U_N ( t) \geq 0 $ and $ \frac{N-1}{N} \le 1 , $ that
$$ \frac1N \sum_{i=1}^N \int_0^t ( U_i ( s- ) - 1 ) d N_i  (s) \le \bar U_N ( 0 ) .$$
We use that $ x- 1 \geq \frac{x}{2} 1_{\{ x \geq 2 \}} - 1_{\{ x \le 1 \}} $ and obtain from this that
\begin{eqnarray*}
 \frac1N \sum_{i=1}^N \int_0^t \frac{U_i ( s- ) }{2} 1_{\{ U_i ( s-) \geq 2 \}} d N_i ( s) &\le& \bar U_N ( 0) + \frac1N \sum_{i=1}^N \int_0^t  1_{\{ U_i ( s-) \le  1  \}} d N_i ( s) \\
&\le&   \bar U_N ( 0)  + \frac{K(t)}{N} .
\end{eqnarray*}
Observing that $ 1 \le \frac{x}{2} 1_{\{ x \geq 2\}} + 1_{\{x \le 2\}} , $ we deduce from the above that
$$  \frac{N(t)}N = \frac1N \sum_{i=1}^N \int_0^t  d N_i ( s) \le \bar U_N (0) +  \frac{K(t)}{N} + \frac{K(t)}{N} , $$
implying the second assertion of (\ref{1.11}).

Since between successive jumps the largest $U_i(t)$ is attracted towards the average potential we can upper bound its position by neglecting the action of the gap junction, implying that
$$ \| U(t) \| \le \| U(0) \| + \frac{ N(t) }{N} ,$$
which, together with (\ref{1.11}), gives \eqref{1.12}, since $ \bar U_N (0 ) \le \| U ( 0 ) \| .$
\end{proof}

{\bf Proof  of Theorem \ref{thm1.1}}
By (\ref{1.12}) we have
$$ \| U (t) \|  \le    2 \| U(0 \| + 2 \frac{ K(T)  }{N} ,$$
for all $ t \le T .$ But the process $K (t) $ is stochastically bounded by a Poisson
process with intensity $ N f(2) .$ Therefore, there exists a constant $K$ such that
$$  P \Big[ | K  (T) | \le  K N \Big] \geq  1  - e^{ - C N T } .$$
This implies (\ref{1.1}). Finally, notice that the above arguments give implicitly
the proof of the existence of the process $U(t), $
since the process can be constructed explicitly, by piecing together
trajectories of the deterministic flow
between successive jump times, once we know that the
number of jumps of the process is finite almost
surely on any finite time interval.  {\hfill $\bullet$ \vspace{0.25cm}}

\section{Proof of Theorem \ref{prop:coupling}.}\label{sec:proofcoupling}
In what follows, $C$ is a constant which may change from one appearance to another.
\subsection{The stopped process}
A technical difficulty in the proof of Theorem \ref{prop:coupling} comes from
the possible occurrence of an anomalously large number of fires
in one of the time steps $[(n-1)\delta,n\delta]$.  To avoid the problem
we stop the process as soon as this happens and prove the theorem for such a stopped process.
We  then conclude by a large deviation estimate for the probability
that the process is stopped before reaching the final time $T$.

Recalling from Proposition \ref{thm1.1bis} and Proposition \ref{lem:use}
that the number of fires in an interval $[(n-1)\delta,n\delta]$
in either one of the two processes is stochastically bounded by a Poisson variable
of intensity $f^* \delta N$ we stop the algorithm defining the coupled process
as soon as  the number of firings in either one of the two processes
exceeds $2f^*\delta N$ in one of the time steps $[(n-1)\delta,n\delta]$.
We call $E$ the event when the process is stopped before reaching the final time. Then
uniformly in the initial datum $ Y^{(\delta)} (0) = U(0 ) =x$,
 \begin{equation}
 \label{Ebis}
P_x (E) \le  2 \frac{T}{\delta} e^{-C N\delta} .
  \end{equation}

By an abuse of notation we denote by the same symbol the stopped processes
and in the sequel, unless otherwise stated, we refer to the stopped process.
We fix arbitrarily $A>0$ and consider the process starting from
$ Y^{(\delta)} (0) = U(0 ) = x $ with $\| x\| \le A .$

Writing $B^* := B + A + 2 f^* T,$ we have by \eqref{1.1mm} for all $t\le T$ and all $n\delta \le T$
\begin{equation}
 \label{E}
 \| U(t)\| \le B^*,\;\; \|Y^{(\delta)}(n\delta)\| \le B^* \quad \mbox{for the stopped processes.}
\end{equation}
It follows that the same bounds hold for the unstopped process
with probability $\ge 1 -  e^{-CN\delta}$.

Thus by restricting to the stopped process we have

\begin{itemize}
\item the firing rate of each neuron is $\le f^*$ and the number of fires of all neurons
in any of the steps $[(n-1)\delta,n\delta]$ is $\le 2f^*\delta N$.

\item  The bounds \eqref{E} are verified and as a consequence the average potentials in the $U$ and $Y^{(\delta)}$ processes are $\le B^*$ so that the gap-junction drift
on each neuron  is $\le \la B^*$.

\end{itemize}

\subsection{Bounds on the increments of $M_n$}

We write $ M_n = M_{ n- 1 } + | A_n \cap \mathcal G_{n-1} |  + | B_n \cap \mathcal G_{n-1} | \le M_{ n- 1 } + | A_n  |  + | B_n \cap \mathcal G_{n-1} |$ where recalling the algorithm given in Subsection \ref{thecoupling} and Definition \ref{goodlabels}
\begin{itemize}
\item $A_n$ is the set of all labels $i$ for which a clock associated to label $i$ rings at least twice during $ [ (n- 1) \delta , n \delta ] .$
\item
$B_n$ is the set of all labels $i$ for which a clock associated to label $i$ rings only once during $ [ (n- 1) \delta , n \delta ] ,$ and it is the clock   $\tau_i^{2}.$
\end{itemize}
We shall prove that (for the stopped processes)
\begin{eqnarray}
\label{M_n1}
&& P \Big[ |A_n | > N  (\delta f^* )^2 \Big] \le e^{ - C N \delta^2} ,
\\&&
\label{M_n2}
  P \Big[  | B_n \cap \mathcal G_{n-1} | > 2 C N \delta \left[  \theta_{n-1} +  \delta \right]  \Big]
  \le e^{ - C N \delta^2 }
    \end{eqnarray}
(recall $C$ is a constant whose value may change at each appearance).

It will then follow that with probability $\geq 1 - 2 e^{ - C N  \delta^2 } $
    \begin{equation}
    \label{eq:429}
M_n \le M_{n-1} + N  (\delta f^* )^2 + 2 C N \delta
\left[  \theta_{n-1} +  \delta \right]\le M_{n-1} + C N \delta
\left[ \theta_{n-1} + \delta \right].
    \end{equation}
Iterating the  upper bound and using that $ n \delta \le T,$ we will then
conclude that
with probability $\geq 1 - 2n  e^{ - C N  \delta^2 } \geq 1 - \frac{C}{\delta }  e^{ - C N  \delta^2 },$ where $C$ depends on $T$,
\begin{equation}
   \label{M_n3}
\frac{M_n}{N} \le C \delta \sum_{k=1}^{n-1} \theta_k + C \delta \le C (\theta_{n-1} + \delta ) \mbox{ for all } n \le \frac{T}{\delta } ,
\end{equation}
having used that, by definition, $\theta_k \le \theta_{n-1} .$

{\bf Proof of \eqref{M_n1}.}

$ |A_n |$ is stochastically upper bounded by $S^*:= \sum_{ i= 1}^N 1_{ \{ N_i^*  \geq 2 \} } ,$ where $ N_1^*, \ldots N_N^*  $ are independent Poisson variables of parameter $f^*\delta$, $f^* = \| f \|_\infty .$ We write $p^* = P ( N_i^*  \geq 2 ) $ and have
   $$
   e^{ - \delta f^* } \frac12 \delta^2 (f^*)^2 \le   p^* \le \frac12  (\delta f^*)^2,\quad
   p^* \approx   \frac{ 1}{2 }\,(\delta f^*)^2  \mbox{ as } \delta \to 0 .
   $$
$S^*$ is the sum of $N$ Bernoulli variables, each with average $p^*$. Then by
the Hoeffding's inequality, we get  \eqref{M_n1}.\\

{\bf Proof of \eqref{M_n2}.}

We shall prove that  the random variable
$ |B_n \cap \mathcal G_{n-1} | $ (for the stopped process)
is stochastically upper bounded by
$\sum_{i= 1}^N \mathbf 1_{\{\bar N_i \ge 1\}}, $ where $ \bar N_i, i = 1 , \ldots , N , $ are independent Poisson variables of parameter $ C ( \theta_{n-1} + \delta ) \delta.$  \eqref{M_n2} will then follow straightly.

We shorthand
   \[
   y:= Y^{(\delta)} ( (n-1) \delta ), \;x:=  U ( (n-1) \delta ),\;\;
y(\delta):=  Y^{(\delta)} ( n\delta ), \;x(t):=  U ( (n-1) \delta+t ),\; t\in [0,\delta] ,
   \]
and introduce independent random times $\tau_i^2$, $i=1,..,N$, of intensity
$|f(y_i)- f(x_i(t)|$, $t\in [0,\delta[$. Then $|B_n|$ is stochastically bounded
by $\sum_{i=1}^N \mathbf 1_{\{ \tau_i^2<\delta \}}$ because  we are neglecting some of the conditions
for being in $B_n$.  We also obviously have
   \[
|B_n \cap \mathcal G_{n-1}| \le \sum_{i=1}^N
\mathbf 1_{\{\tau_i^2<\delta, i\in \mathcal G_{n-1}\}} \quad \text{stochastically.}
   \]
To control the right hand side we
bound
   \[
    |f(x_i(t)) - f(y_i) | \le \|f \|_{Lip} |x_i(t)-y_i|,
    \]
$ \|f \|_{Lip} $ the Lipschitz constant of the function $f.$
Denote by $N_j(s,t)$ the number of spikes of $U_j(\cdot)$
in the time interval $[s,t]$, then analogously to \eqref{eq:flow},
     $$
|x_i (t) - y_i | \le  | x_i  - y_i | +\int_0^t \lambda e^{ - \lambda ( t- s) } |\bar x (s) - x_i | ds + \frac{1}{ N} \sum_{ j \neq i } N_j ( [ (n-1)\delta,(n-1)\delta+ t ] ) .
    $$
We have
$ | \bar x (s)  - x_i | \le B^* $ and
$\sum_{ j \neq i } N_j ( [ (n-1)\delta,(n-1)\delta+ t ] )\le 2 f^* \delta N$
because we are considering the stopped process.  Then if
$ i \in \mathcal G_{n-1} $,
    $$
|x_i (t) - y_i | \le \theta_{n-1} + B^* \delta +  2 f^* \delta
   $$
and therefore
    $$
|f(x_i(t)) - f(y_i) | \le \|f \|_{Lip} \left(  \theta_{n-1}
+ B^* \delta + 2 f^* \delta \right) \le C ( \theta_{n-1} + \delta )
    $$
 so that
    \[
    \sum_{i=1}^N
\mathbf 1_{\{\s_i<\delta, i\in \mathcal G_{n-1}\} } \le    \sum_{i=1}^N
\mathbf 1_{\{ \bar N_i \geq 1 \}} \quad \text{stochastically,}
   \]
where the $\bar N_i$ are independent Poisson random variables of intensity $C ( \theta_{n-1} + \delta )$. This proves \eqref{M_n2}.

\vskip.5cm

\subsection{Bounds on  $\theta_n$}

The final bound on  $\theta_n$ is reported in \eqref{eq:538fourth} at the end of
this subsection.  We start by characterizing the elements $i \in {\cal G}_{n}$
as $i \in {\cal G}_{n- 1} \cap (C_n \cup F_n )$ where:

\begin{enumerate}
\item $C_n$ is the set of all labels $i$ for which a clock associated to label $i$ rings only once during $ [ (n- 1) \delta , n \delta ] ,$ and it is a clock  $\tau_i^{1}$.
\item
$F_n$ is the set of all labels $i$ which do not have any jump during $ [ (n- 1) \delta , n \delta ] .$
\end{enumerate}

In other words, we study labels $i$ which are good at time
$ (n-1) \delta$ and which stay good at time $ n \delta $ as well.
We shall use in the proofs the following formula for the potential
$U_i(t)$ of a neuron which does not fire in the interval $[t_0,t]$:
    \begin{equation}
   \label{formule}
U_i(t) = e^{-\la(t-t_0)}U_i(t_0) + \int_{t_0}^t \lambda e^{-\la(t-s)}\{\bar U(s) ds +  \frac 1{\la N} dN(s)\} ,
   \end{equation}
$N(t)$  denoting the total number of fires till time $t$.  For the $Y^{(\delta)}$ process
we shall instead use \eqref{eq:flow} and the expressions thereafter.

$\bullet$\;\; Labels  $i\in C_n \cap {\cal G}_{n- 1}$.

For such labels $i$ there is a random time $ t \in [(n-1)\delta,n\delta[$ at which
a $\tau_i^{1 } $ event happens. Then by \eqref{formule}
    $$
U_i (n \delta ) = \int_t^{ \delta }
\lambda e^{ - \lambda ( \delta  - s)} \bar U_N (s) ds +
e^{ - \lambda  \delta  }\frac{1}{N} \int_t^\delta e^{\lambda s} d   N (s)   ,
   $$
because $U_i(t^+)=0$.  Since we are considering the stopped process, $\bar U (\cdot) \le B^*$
and $N(n\delta)-N((n-1)\delta) \le 2f^*\delta N$ so that $U_i (n \delta ) \le C \delta$.
In the same way, $ Y^{(\delta)}_i (n \delta )  \le C \delta ,$
and therefore
   \begin{equation}
   \label{eq:din2}
 D_i (n  ) = |U_i ( n \delta ) - Y^{(\delta)}_i ( n \delta ) | \le  C \delta ,\quad \text{for the stopped process} .
   \end{equation}
Notice that the  bound does not depend on $D_i(n-1)$.

$\bullet$\;\; Labels $i\in F_n \cap {\cal G}_{n- 1}$.

This means that $i$ is good at time $ (n-1) \delta $ and does not
jump, neither in the $U$ nor in the $Y^{(\delta)}$ process.
Let $ U( (n-1) \delta ) = x $ and $ Y^{(\delta)} ( (n-1) \delta ) = y.$
By \eqref{formule} and \eqref{319} $|U_i ( n \delta ) - Y^{(\delta)}_i ( n \delta ) | = D_i(n)$ is bounded by
    \begin{multline}
    \label{eq:538}
D_i (n) \le  e^{-\lambda \delta }|  x_i - y_i |  + (1 - e^{-\la \delta}) |\bar x-\bar y|
 +
 \int_{(n-1)\delta}^{n\delta} \lambda e^{- \lambda (n\delta -t )} |\bar U_N (t) - \bar U_N(0) ]  dt \\
+ \frac{1}{N} |\int_{(n-1)\delta}^{n\delta}  e^{-\la(n\delta -t)}  dN(t)  - Nq|
\end{multline}
where $Nq$ is the total number of fires in the process
$Y^{(\delta)}$ in the step from $(n-1)\delta$ to $n\delta$.

We bound the right hand side of \eqref{eq:538} as follows. We have $e^{-\lambda \delta }|  x_i - y_i | \le \theta_{n-1}.$ Moreover,
$$
   (1 - e^{-\la \delta}) |\bar x-\bar y| \le \la \delta  \Big(\theta_{n-1} + B^*\frac{M_{n-1}}N\Big),\quad
   \text{$B^*$ as in \eqref{E}},
$$
and
$$
   \int_{(n-1)\delta}^{n\delta} \lambda e^{- \lambda (n\delta -t )} |\bar U_N (t) - \bar U_N(0) ]  dt \le
   \la \delta \frac 1N N((n-1)\delta,n\delta) ,
$$
where $N((n-1)\delta,n\delta)= N(n\delta)-N((n-1)\delta)$.  Writing
   \[
\int_{(n-1)\delta}^{n\delta}  e^{-\la(n\delta -t)}  dN(t) = N((n-1)\delta,n\delta) +   \int_{(n-1)\delta}^{n\delta} \{ e^{-\la(n\delta -t)}-1\}  dN(t) ,
   \]
we bound the last term on the right hand side of \eqref{eq:538}    as
   \[
   \frac{1}{N} |\int_{(n-1)\delta}^{n\delta}  e^{-\la(n\delta -t)}  dN(t)  - Nq| \le \frac 1N\Big(
   |N((n-1)\delta,n\delta)-Nq| + \la\delta   N((n-1)\delta,n\delta)\Big) .
   \]
Collecting all these bounds we then get
    \begin{multline*}
D_i (n) \le  \theta_{n-1}(1+\la \delta  ) + \la\delta  B^*\frac{M_{n-1}}{N}+
2\la \delta   \frac 1N N((n-1)\delta,n\delta) \\
+\frac 1N
   |N((n-1)\delta,n\delta)-Nq| ,
\end{multline*}
and since we are considering the stopped  process
    \begin{multline}
    \label{eq:538bis}
D_i (n) \le  \theta_{n-1}(1+\la \delta  ) + \la\delta  B^* \frac{M_{n-1}}{N}+
2\la \delta  2f^*\delta
+\frac 1N
   |N((n-1)\delta,n\delta)-Nq| .
\end{multline}
By the definition of the sets $A_n ,\ldots , F_n$ we have
   \begin{equation}
   \label{eq:539a}
|N((n-1)\delta,n\delta) -Nq| \le \sum_{j\in A_n} N_j((n-1)\delta,n\delta) +|B_n| .
   \end{equation}
What follows is devoted to the control of the rhs of (\ref{eq:539a}). We start with $|B_n|.$
With probability $\ge 1- e^{-C\delta^2N}$
   \begin{equation}
   \label{eq:539b}
|B_n| \le |B_n \cap \mathcal G_{n-1}|+ |B_n \cap M_{n-1}| \le 2C N\delta^2 +2C N \delta \theta_{n-1}+ |M_{n-1}| 2 f^*\delta ,
   \end{equation}
having used \eqref{M_n2} and that the number of neurons among those
in $M_{n-1}$ which fire in a time $\delta$ is bounded by a Poisson variable
of intensity $f^* \delta |M_{n-1}|$. Moreover,
   \begin{multline}
   \label{eq:539c}
P\Big[ \sum_{j\in A_n} N_j((n-1)\delta,n\delta) \ge 4 (f^*\delta)^2 N\Big]
\le P\Big[ \sum_{j\in A_n} N_j((n-1)\delta,n\delta)\\ \ge 4 (f^*\delta)^2 N; |A_n| \le
(f^*\delta)^2 N \Big] + P\Big[ |A_n| >
(f^*\delta)^2 N \Big] .
   \end{multline}
The last term is bounded using \eqref{M_n1}.

We are now going to bound the first term in (\ref{eq:539a}). For that sake, let $A \subset \{1,..,N\}$, $|A| \le
(f^*\delta)^2 N ,$ then
   \[
P\Big[ \sum_{j\in A_n} N_j((n-1)\delta,n\delta)\\ \ge 4 (f^*\delta)^2 N\;|\; A_n =A  \Big]
\le P^*\Big[ \sum_{j\in A}( N^*_j-2) \ge  2 (f^*\delta)^2 N  \Big] ,
   \]
where $P^*$ is the law of independent Poisson variables $N^*_j$, $j\in A$, each one of parameter $f^*\delta$
and
conditioned on being $N^*_j\ge 2$.  Thus the probability that $N_j^*-2 = k$ is
   \[
 P^*[N^*_j-2 = k ]=  Z_\xi^{-1} \frac{\xi^k}{(k+2)!},\quad Z_\xi = \xi^{-2} \Big(e^\xi - 1 -\xi\Big),\quad
 \xi = f^*\delta .
   \]
Denote by $X_j$ independent Poisson variables of parameter $\xi .$ Then it is easy to see that
$N^*_j-2 \le X_j$ stochastically for $\xi$ small enough, hence
for $\delta$ small enough. Notice that
$X=\sum_{j\in A} X_j$ is a Poisson variable of parameter $|A| \xi \le (f^*\delta)^2 N f^*\delta$ having expectation $E^* (  X) \le (f^*\delta)^2 N$ for $\delta $ small. As a consequence we may conclude that
   \[
 P^*\Big[ \sum_{j\in A}( N^*_j-2) \ge  2 (f^*\delta)^2 N  \Big] \le P^*\Big[X \ge  2 (f^*\delta)^2 N  \Big]  \le e^{-CN \delta^{2}} .
   \]
In conclusion for   $i\in F_n \cap {\cal G}_{n- 1}$:
    \begin{equation}
    \label{eq:538ter}
D_i (n) \le  \theta_{n-1}(1+ C \delta ) + C\delta  \frac{M_{n-1}}{N}+ C\delta^2
  \end{equation}
with probability $\ge 1- e^{-C\delta^2 N}$.  Together with \eqref{eq:din2} this proves
that with probability $\ge 1- e^{-C\delta^2 N}$
    \begin{equation}
    \label{eq:538fourth}
\theta_n \le \max\{ C\delta; \theta_{n-1}(1+ C \delta ) + C\delta  \frac{M_{n-1}}{N}+ C\delta^2\} .
  \end{equation}

\vskip.5cm

\subsection{Iteration of the inequalities}
By \eqref{M_n3}, $\dis{\frac{M_n}{N} \le C (\theta_{n-1} + \delta )}$ for all $n\delta \le T$
with probability $\ge 1 - \frac T \delta e^{-CN\delta^2}$.  By \eqref{eq:538fourth},
with probability $\ge 1- \frac T \delta e^{-C\delta^2 N}$ we have
    \[
\theta_n \le \max\{ C\delta; \theta_{n-1}(1+ C \delta ) + C\delta  \frac{M_{n-1}}{N}+ C\delta^2\} .
  \]
Thus
     $$
\theta_n \le \max \Big( C \delta ,\left[ 1 + C \delta \right] \theta_{n-1} + C \delta^2 \Big) ,
     $$
since $ \theta_{n-2} \le \theta_{n-1} .$ Iterating this inequality we obtain
\begin{multline*}
 \theta_n \le   C \sum_{ k=0}^{n-1} \left[ 1 + C \delta \right]^k  \delta^2 + (1 + C \delta)^n  C \delta
=  C \frac{ \left[ 1 + C \delta \right]^n - 1 }{ C \delta } \delta^2  + (1 + C \delta)^n  C \delta \\
\le  C e^{ C T } \delta  ,
\end{multline*}
where we have used once more that $ n \delta \le T .$ Hence
$$ \theta_n \le C \delta $$
for all $ \delta \le \delta_0, $ with probability $\geq 1 - \frac{C}{\delta^2 } e^{ - C N \delta^2 }.$ This finishes the proof of Theorem \ref{prop:coupling}.  {\hfill $\bullet$ \vspace{0.25cm}}

\section{Proof of Theorem \ref{theo:2}.}\label{sec:prooftheo2}
The proof is by induction on $n.$ Firstly, \eqref{A.10} holds for $ n = 0, $ since $B_0 = 0 $ and $x^N_1, \ldots , x^N_N $ i.i.d.\ according to $\psi_0 (x) dx .$  We then suppose that \eqref{A.10} holds for all $j\le n$.
We condition on $\mathcal F_n$ and introduce $\dis{G_n=\bigcap_{j\le n} \{d_j ( Y^{(\delta)} ,\rho^{(\delta)}) \le \kappa_j
\}}.$ Then
\begin{multline}
 S_{x}^{(\delta,N, \lambda )}\Big[d_{n+1}( Y^{(\delta)} ,\rho^{(\delta)}) > \kappa_{n+1}\Big]  \\
\le
 S_{x}^{(\delta,N, \lambda )}\Big( \mathbf 1_{G_n}  S_{x}^{(\delta,N, \lambda )}\Big[d_{n+1}( Y^{(\delta)} ,\rho^{(\delta)}) > \kappa_{n+1}\;|\; \mathcal F_n\Big]
 \Big)\\
 + n c_1(n) e^{ - c_2 N^\gamma  } .
       \end{multline}
Therefore, we need to prove that for some constant $c$
     \begin{equation}
      \label{A.10bis}
 S_{x}^{(\delta,N, \lambda )}\Big[d_{n+1}( Y^{(\delta)} ,\rho^{(\delta)}) > \kappa_{n+1}\;|\; \mathcal F_n\Big]  \le   c  e^{ - c_2 N^\gamma  }, \quad \text{on $G_n.$}
       \end{equation}
From the conditioning we know that
$d_j( Y^{(\delta)} ,\rho^{(\delta)}) \le \kappa_j$
for all $j\le n$; we know also the value of $ Y^{(\delta)} ( n \delta ),$ say $ Y^{(\delta)} ( n \delta )= y ,$ we know
the location of the edges $R'_j, j\le n , $ and we know which are the good and the bad intervals at time $n\delta$.

{\bf Consequences of being in $G_n$}

The condition to be in $G_n$ does not only allow to control the quantities directly involved
in the definition of $d_n$ but also several other quantities.
The first one is the difference $|R'_n-R_n|$.  Indeed, recalling
\eqref{R_n} and \eqref{R'_n},
\begin{eqnarray}
    \label{R'_n-R_n}
    |R'_n-R_n| &\le & e^{- \lambda \delta }|R_{n-1}'-R_{n-1}| + | V((n-1)\delta) -p_{(n-1) \delta }^{(\delta )} | \delta
    \nonumber \\
&& \quad \quad \quad \quad \quad \quad \quad \quad \quad \quad \quad \quad  +(1-e^{-\lambda \delta}) | \bar Y^{(\delta)} ((n-1)\delta) -  {\bar\rho}^{(\delta )}_{ (n-1)  \delta }|
    \nonumber .
  \\
     &\le &e^{- \lambda \delta }|R_{n-1}'-R_{n-1}| + \kappa_{n} \ell
    + \lambda\delta  \kappa_{ n-1} \ell   .
\end{eqnarray}
Iterating this argument yields
$$  |R'_n-R_n|
\le  \Big(\sum_{j=1}^{n} \kappa_{j}\Big)(1+\lambda \delta)\ell .$$
Writing $I_{i,n}= [a_{i,n},b_{i,n}]$ and $I'_{i,n}= [a'_{i,n},b'_{i,n}],$ we obtain in particular
    \begin{equation}
    \label{a-a'}
    |a_{i,n}-a'_{i,n}| = |b'_{i,n}-b_{i,n}| \le  \Big(\sum_{j=1}^{n} \kappa_{j}\Big)(1+\lambda \delta)\ell .
    \end{equation}

We also get a bound on the increments of the number of bad intervals.
Recalling \eqref{numbbad} for notation we have indeed
        \begin{equation}
        \label{B_n-B_n-1}
B_{n} \le B_{n-1} + 1+ \frac{ | R_{n} - R'_{n} |}{\ell } .
       \end{equation}

We finally have bounds on $ N'_{i,n}  .$ Firstly we suppose that $I_{i,n}$ is a good interval such that $ I_{ i, 0 } \subset \R_+ .$
Then $N'_{i,n} \le N'_{i,0}$, whence for $N$ large enough, since $ N {\rm w}_i = N_{i, 0} , $
        \begin{equation}
        \label{N'_i,n}
 N'_{i,n} \le N_{i,0} + \kappa_0 {\rm w}_i   N \ell \le  (1+\kappa_0 \ell) N {\rm w}_i \le 2N {\rm w}_i .
      \end{equation}
We also have a lower bound.  By \eqref{3.5}, $N_{i,n} \geq N_{i,0} e^{-f^* \delta n}$, hence
        \begin{equation}
        \label{N'_i,nlower}
N'_{i,n} \geq N_{i,n}- \kappa_n  {\rm w}_i N \ell \geq  {\rm w}_i N \Big(e^{-f^* T}-\kappa_n \ell \Big)
\geq c {\rm w}_i N \geq  c \ell^3 N = c r^3 N^{ 1 - 3 \alpha } ,
     \end{equation}
for $N$ large enough.

Now consider a good interval $I_{i,n} \subset \R_+$ such that $ I_{ i, 0 } \subset \R_- .$ Then there exists $k \le n $ such that $ I_{i, k- 1 } \subset \R_- $ and $I_{i, k } \subset \R_+ .$ Recalling the definition
of $d_k $ in (\ref{eq:dn}) and the definition (\ref{318})
we notice that $ d_k \geq 1/N  , $ if $ N \delta V (k \delta) \geq 2 ,$
i.e.\ in case that at least two neurons spike.
At step $k,$ the number of neurons falling into the interval $ I_{i, k }$ is upper bounded by $ \frac{ \ell }{d_k } + 1 ,$ if $   N \delta V (k\delta)  \geq 2,$
otherwise, there is at most one neuron falling into it.
In both cases, this yields the upper bound $ \ell N + 1 $
for the number of neurons falling into the interval. After time $k,$ neurons originally in $ I_{i,k } $ can only disappear due to spiking. Hence,
\begin{equation}\label{eq:678bis}
 N'_{i,n} \le N'_{i, k } \le N ( \ell + N^{ - 1 } ) \le C N {\rm w}_i ,
\end{equation}
 by definition of ${\rm w}_i.$ In order to obtain a lower bound, we first use that
$$ N'_{i,n} \geq N_{i,n}- \kappa_n  {\rm w}_i N \ell  .$$
Since $ I_{i, k -1} \subset \R_-, $ we have that $ \rho_{k\delta}^{(\delta)} \equiv \rho_{k \delta}^{(\delta)} ( 0 ) $ on $I_{i, k },$ hence
$ N_{i, k } =  N \rho_{ k \delta}^{(\delta ) } ( 0) \ell e^{ - \lambda \delta k}.$ Using Proposition \ref{prop:6} and (\ref{3.5}),
$$ N_{i, n } \geq N_{i, k } e^{ - f^* \delta (n -k ) } \geq  N \rho_{ k \delta}^{(\delta ) } ( 0) \ell e^{ - \lambda T} e^{ - f^* T } \geq C N \psi_0 ( 0 ) \ell  = C N {\rm w}_i   , $$
where $C$ depends on $T,$ which allows to conclude as above.

In case that $I_{i,n}$ is a bad interval it is easy to see that the upper bound
        \begin{equation}
        \label{N'_i,nbad}
 N'_{i,n } + N_{i, n} \le C N \ell
     \end{equation}
holds.

{\bf Expected fires in good intervals}

Recall that we are working on $G_n$ and conditionally on $Y^{(\delta)} (n \delta ) = y.$ Using
 (\ref{eq:424}) and (\ref{eq:555bis}), we write
     \begin{equation}
V ( n \delta)  = \frac{1}{\delta N} \sum_{ i } \Delta_i , \mbox{ where } \Delta_i = \sum_{ j : y_j \in I'_{i,n } } \Phi_j (n  ) ,
      \end{equation}
and call $\langle \Delta_i \rangle $ its conditional expectation (given $\mathcal F_n,$ and hence given
that $Y^{(\delta)} ( n \delta ) = y $). Then
           $$
\langle \Delta_i \rangle = \sum_{ j : y_j \in I'_{i,n } } \left( 1 - e^{ -\delta  f( y_j)} \right) .
           $$
Write $I'_{i,n}=[a'_{i,n},b'_{i,n}].$ Since $f$ is non decreasing, we have
    $$
\langle \Delta_i \rangle \le N'_{i,n} (1- e^{-\delta f(b'_{i,n})}) \le N'_{i,n} (1- e^{-\delta f(a'_{i,n})}) + N'_{i,n}|e^{-\delta f(b'_{i,n})}-e^{-\delta f(a'_{i,n})}| .
    $$
Moreover, $f(b'_{i,n}) \le f( a'_{i,n}) + \| f\|_{Lip}  \, \ell ,$ which implies that
    $$
|e^{-\delta f(b'_{i,n})}-e^{-\delta f(a'_{i,n})}|\le \delta C\ell .
   $$

Suppose first that $I'_{i,n}$ is good so that $|N'_{i,n}- N_{i,n}| \le \kappa_n {\rm w}_i N\ell $.  Then by \eqref{N'_i,n} and \eqref{eq:678bis},
   \begin{eqnarray*}
\langle \Delta_i \rangle& \le &  \{N_{i,n} + \kappa_n   {\rm w}_i N\ell \}(1- e^{-\delta f(a'_{i,n})})+\delta C\ell  N{\rm w_i}
\\
& \le&  N_{i,n} (1- e^{-\delta f(a'_{i,n})})  + C (\kappa_n +1 )  \delta {\rm w}_i N\ell .
    \end{eqnarray*}
Write $I_{i,n}=[a_{i,n},b_{i,n}],$ so that by \eqref{a-a'},
$|a'_{i,n}-a_{i,n}| \le K_n \ell,$ where $K_n = (\sum_{j=1}^n \kappa_n ) (1 + \lambda \delta) .$  Then
$$
\langle \Delta_i \rangle \le   N_{i,n} (1- e^{-\delta f(a_i)})  + \big(K_n +C(1+\kappa_n)\big)\delta{\rm w}_i N\ell .
$$
Since $f$ is non decreasing and $N_{i,n} = N \int_{I_{i,n}} \rho^{(\delta)}_{n\delta}(x)\, dx ,$
     \begin{equation}
     \label{A.8}
\langle \Delta_i \rangle \le N \int_{I_{i,n}} \rho^{(\delta)}_{n\delta}(x) (1- e^{-\delta f(x)})dx +
  \big(K_n +C(1+\kappa_n)\big)\delta{\rm w}_i N\ell .
      \end{equation}
An analogous argument gives
     \begin{equation}
     \label{A.8lower}
\langle \Delta_i \rangle \ge N \int_{I_{i,n}} \rho^{(\delta)}_{n\delta}(x) (1- e^{-\delta f(x)})dx -
  \big(K_n +C(1+\kappa_n)\big)\delta{\rm w}_i N\ell .
      \end{equation}

{\bf Fires fluctuations in good intervals}

Hoeffding's inequality implies that for any $ b > 0, $
     \begin{equation}
     \label{eq:665}
 P \left[ | \Delta_i - \langle \Delta_i \rangle | \geq (N'_{ i,n })^{b + \frac12} \right] \le 2 e^{- 2 (N'_{ i,n })^{2b} } .
      \end{equation}
We introduce the contribution of $I_{i,n}$ to $ p^{(\delta)}_{n\delta }$
   $$
 p^{(\delta)}_{i,n\delta }= \frac{N}{\delta} \int_{I_{i,n}} \rho^{(\delta)}_{n\delta}(x) (1- e^{-\delta f(x)})dx .
   $$
We then use \eqref{N'_i,n}, \eqref{N'_i,nlower} and \eqref{eq:678bis}  together with \eqref{A.8} and \eqref{eq:665}
to get
\begin{multline}
\label{A.9}
P\left[\bigcap_{I_{i,n}\;{\rm good}}\left\{\Delta_i \le \delta p^{(\delta)}_{i,n \delta }+  \big(K_n +C(1+\kappa_n)\big)\delta{\rm w}_i N\ell
+\big(C N {\rm w}_i\big)^{\frac 12 + b} \right\}\right] \\
\geq 1- 2 m_n   e^{-2 ( c r^3 N^{1-3\alpha})^{2b}}  ,
\end{multline}
where $m_n $ is an upper bound on the number of good intervals which can be upper bounded by
    $$
m_n \le \left( \frac{R_n }{e^{ - \lambda n \delta } \ell} \vee \frac{R'_n}{e^{ - \lambda n \delta }  \ell} \right) + 1 \le C \left( \frac{R_n}{\ell } + [ \sum_{ j = 1 }^n \kappa_j ] \right)  +1  \le C N^\alpha ,
   $$
because $ R_n \le R_0 + c^* T $ and $n \delta \le T.$ As a consequence, the right hand side of (\ref{A.9}) can be lower bounded by
$ 1 - C N^\alpha e^{-C ( N^{1-3\alpha})^{2b}} .$ By an analogous argument
     \begin{multline}
     \label{A.10b}
 P\left[ \bigcap_{I_{i,n}\;{\rm good}}\{\Delta_i \geq \delta p^{(\delta)}_{i,n \delta }- \big(K_n +C(1+\kappa_n)\big)\delta{\rm w}_i N\ell
-\big(2 N {\rm w}_i\big)^{\frac 12 + b} \}\right] \\
 \geq 1-C N^\alpha    e^{- C (  N^{1-3\alpha})^{2b}} .
      \end{multline}
Now we choose $b  $ and $ \alpha  $ sufficiently small such that for $N$ large enough, $(C N {\rm w}_i )^{\frac 12 + b} \le C\delta(1+\kappa_n)\big){\rm w}_i N\ell .$
Then
     \begin{multline}
     \label{A.11}
 P\left[\bigcap_{I_{i,n}\;{\rm good}}\left\{|\Delta_i- \delta p^{(\delta)}_{i,n \delta }| \le \big(K_n +2C(1+\kappa_n)\big)\delta{\rm w}_i N\ell
 \right\}\right] \\
 \geq 1- C N^\alpha   e^{-C (  N^{1-3\alpha})^{2b}}  \geq 1 - C e^{ - C N^\gamma } ,
    \end{multline}
where $ \gamma = (1 - 3 \alpha ) 2 b $ and $C$ a suitable constant.

{\bf The bounds on $V (n\delta) $ and on $B_{n+1}$}

By \eqref{N'_i,nbad} and \eqref{A.11}, with probability $\geq 1 - C e^{ - C N^\gamma } ,$
    \begin{equation}
    \label{V_n}
| \delta V (n \delta)   - \delta p_{n\delta}^{(\delta ) }    | \le  \left[ \sum_{I_{i,n}\;{\rm good}} \left (K_n +2C(1+\kappa_n)\right)\delta{\rm w}_i \ell \right]  + \frac{C N\ell }N B_n
 \le \kappa'_{n +1 } \ell .
    \end{equation}
Hence, we have proven the desired assertion for $ V_n $ at time $ (n+1 ) \delta .$

To bound $B_{n+1},$ the number of bad intervals at time $(n+1) \delta ,$ we use the first
inequality in \eqref{R'_n-R_n} with $n\to n+1$ together with \eqref{V_n}, so that by \eqref{B_n-B_n-1}
        $$
B_{n+1} \le B_n + 1+ \frac{ | R_{n+1} - R'_{n+1} |}{\ell }  \le B_n + 1+ \Big(\sum_{j=1}^{n} \kappa_{j}+\kappa'_{n+1}+ \kappa_n\Big)(1+\lambda \delta),
        $$
whence the assertion concerning $B_{n+1}$.

{\bf Bounds on $|N'_{i, n+1} - N_{i, n+1} |$}

Let $I_{i,n}$ be a good interval at time $n\delta$ which is contained in $\mathbb R_+$.  Then it is good also at time
$(n+1)\delta$ and we have
$$
N'_{i, n+1} = \sum_{ j : y_j \in I'_{n,i} } (1 - \Phi_j (n ) ) = N'_{i, n } - \Delta_i  \mbox{ and } N_{i, n+1} = N_{i, n }- \delta p^{(\delta)}_{i,n\delta } .
$$
Thus
   \[
\frac{|N'_{i,n+1}- N_{i,n+1}| }{ {\rm w}_i N\ell} \le \kappa_n + \frac{|\Delta_i- \delta p^{(\delta)}_{i,n\delta }|
}{ {\rm w}_i N\ell} ,
    \]
and the desired bound follows from \eqref{A.11}.

It remains to consider a good interval $ I'_{i, n+1} $ such that $I'_{i,n} \subset \mathbb R_-$ (and hence
also $I_{i,n} \subset \mathbb R_-$).  Thus $ I'_{i, n+1} $  consists entirely of ``new born'' neurons which arise due to firing events where the energies are reset to $0.$
For such an interval,
$$\frac{ N'_{i, n+1}}{N \ell } \in [ \frac{1}{N d_n} e^{ - \lambda \delta n} , \frac{1}{N d_n} e^{ - \lambda \delta n } + 1 ] .$$
But, recalling the definition of $d_n$ in (\ref{eq:dn}) and of $\rho_{(n+1) \delta }^{(\delta)} ( 0 ) $ in (\ref{eq:conditionaubord}),  by continuity of $ (u,p ) \to \frac{ p \delta }{ p \delta + (1- e^{ - \lambda \delta} )u}  ,$ we have
     \begin{equation}
 |\frac{1}{N d_n} - \rho^{(\delta)}_{(n+1)  \delta } (0) |\le C \kappa_n \ell .
       \end{equation}
Since $\rho^\delta_{(n+1) \delta } (x) 1_{ I_{i, n+1} } (x) \equiv \rho_{(n+1) \delta}^{(\delta)} (0) $ on this interval, this implies that also for such intervals,
$$ \frac{1}{N \ell } | N' _{i, n+1} - N_{i, n+1} | \le C \kappa_n \ell = C \kappa_n {\rm w}_i , $$
by definition of ${\rm w}_i .$ This concludes the bound of $| N' _{i, n+1} - N_{i, n+1} |$.

The bound on  $| \bar Y^{(\delta)}( (n+1)\delta) - {\bar \rho}_{(n+1) \delta }^{(\delta ) }| $ follows from
the bounds on $|N'_{i,n+1}- N_{i,n+1}|$ and $B_{n+1}$; details are omitted. This concludes the proof of Theorem \ref{theo:2}.
 \hfill $\bullet $

\section{Proof of Theorem \ref{theo:main} for general firing rates.}

Let $f$, $T$, $A$, $B$ as in Theorem \ref{thm1.1}, $x^{N}$ the initial state
of the neurons as in Theorem \ref{theo:main} and such that $\|x^{N}\|\le A$.
Let $\psi$ be a bounded continuous
function on  $  D ( [0, T ] , {\cal S}') $.  We need to prove that
    \[
 \lim_{N\to \infty} {\cal P}^N_{[0, T ] }(\psi)= \psi(\rho)
    \]
where ${\cal P}^N_{[0, T ] }(\psi)$ is the expected value of $\psi$ under the law of
$(\mu_{ U^N})_{ [0, T ] } $ when the process $U^N$ starts from $x^{N}$
and $\psi(\rho)$ is the value
of $\psi$ on the element $\rho:=(\rho_t dx)_{t\in [0,T]}$ of $  D ( [0, T ] , {\cal S}') $.

Let $\mathbf 1_{\mathcal U}$ be the characteristic function of the event
$\{\| U^N(t) \|\le B, t\in [0,T]\}$.  Then by Theorem \ref{thm1.1}
   \begin{equation}\label{A.107}
   \lim_{N\to \infty}\big| {\cal P}^N_{[0, T ] }(\psi) - {\cal P}^N_{[0, T ] }(\psi \mathbf 1_{{\mathcal U}})\big| =0 .
   \end{equation}
By an abuse of notation we call ${\cal P}^{*,N}_{[0, T ] }$ the law of the process with a firing rate $f^*(\cdot)$
which satisfies Assumption  \ref{ass:2} and coincides with $f$ for $x\le B$.  Then
   \begin{equation}
   \label{A.108}
   {\cal P}^N_{[0, T ] }(\psi \mathbf 1_{\mathcal U})  = {\cal P}^{*,N}_{[0, T ] }(\psi \mathbf 1_{\mathcal U}) .
   \end{equation}
Since we have proved Theorem  \ref{theo:main} under  Assumption  \ref{ass:2}, we have convergence
for the process
with rate $f^*(\cdot)$ to a limit density that we call $\rho^*=(\rho^*_t)_{t\in [0,T]}$, so that
   \begin{equation}
   \label{A.110}
 \lim_{N\to \infty}  {\cal P}^{*,N}_{[0, T ] }(\psi 1_{\mathcal U}) =   \psi(\rho^* 1_{\mathcal U}) .
   \end{equation}
As a consequence of \eqref{A.107} and \eqref{A.108},
    \[
\lim_{N\to \infty}  {\cal P}^N_{[0, T ] }(\psi) = \psi(\rho^* 1_{\mathcal U} ) .
    \]
By the arbitrariness of $\psi$, $\rho^*=\rho^* \mathbf 1_{{\mathcal U}}.$ Indeed, taking $ \psi (\omega)= \sup \{ \omega_t ( 1 ) , t \le T \} \wedge 1 , $ we have $\lim_{N\to \infty}  {\cal P}^N_{[0, T ] }(\psi)  \equiv 1 , $ which implies that $\rho^*$ must have
support in $[0,B] .$ As a consequence,
  \[
\lim_{N\to \infty}  {\cal P}^N_{[0, T ] }(\psi) = \psi(\rho^* 1_{\mathcal U} ) = \psi ( \rho^* ) ,
    \]
and the limit $\rho^* $ is equal to the solution of the equation with the true firing rate $f$. This concludes the proof of the theorem.

\section*{Acknowledgments}

We thank M. Cassandro and D. Gabrielli for their collaborative
participation to the first stage of this work. Lemma \ref{lem:nicolas} is due to discussions of EL with Nicolas Fournier. We also thank B. Cessac,
P. Dai Pra and C. Vargas for many illuminating discussions.

This article was produced as part of the activities of FAPESP Research,
Dissemination and Innovation Center for Neuromathematics (grant
2013/07699-0, S.\ Paulo Research Foundation). This work is part of USP project ``Mathematics, computation, language
and the brain",
USP/COFECUB project ``Stochastic systems with interactions of variable
range'' and CNPq project ``Stochastic modeling of the brain activity"
(grant 480108/2012-9).  ADM is partially supported by PRIN 2009 (prot. 2009TA2595-002).
AG is partially supported by a CNPq fellowship
(grant 309501/2011-3). AG and EL thank GSSI for hospitality and
support.

\bibliography{biblio}

\end{document}